\documentclass[12 pt]{article}
\usepackage[english]{babel}
\usepackage[utf8]{inputenc}
\usepackage{amsmath}
\usepackage{color,latexsym,amsthm,amsmath,amssymb,path,enumitem,url,bbm,subcaption,geometry,bbm}
\usepackage{bm}

\usepackage[pdftex]{graphicx}

\usepackage[utf8]{inputenc}
\newtheorem{theorem}{Theorem}[section]

\newtheorem{lemma}[theorem]{Lemma}

\newcommand{\parag}[1]{\vspace{2mm}

\noindent{\bf #1} }

\newcommand{\RR}{\mathbb R}

\newcommand{\pg}{{\mathsf{pg}}}
\newcommand{\gr}{{\mathsf{gr}}}
\newcommand{\pgk}{{\mathsf{pg}^k}}

\newcommand{\pts}{\mathcal{P}}
\def\hv{{\hat{v}}}
\def\hp{{\hat{p}}}

\newcommand{\pot}{{\mathsf{pot}}}

\title{The Minimum Number of Plane Graphs for Sets with Small Hulls}

\author{
Alice Chen\thanks{{\sl ploris085@gmail.com}}
\and
Tara Saini\thanks{{\sl sainitara09@gmail.com}}
\and 
Adam Sheffer\thanks{CUNY: Baruch College, NY, USA. This work is partially supported by by PSC/CUNY award 67511-00 55.
{\sl adamsh@gmail.com}.}
\and
Angela Zhang\thanks{{\sl angelazhang2026@gmail.com}}
}

\date{}

\begin{document}
\maketitle
\begin{abstract} 
Let $\pts$ be a set of $n$ points in $\RR^2$, with a convex hull of size $O(n/\log n)$. 
We prove that $\Omega(12.24^n)$ plane graphs can be drawn on $\pts$, the first non-trivial bound for this problem. 
We also show that a random plane graph, uniformly chosen from the set of all plane graphs of $\pts$, has at most $n/12.24$ isolated vertices.
This improves upon a previous bound of $n/10.18$.

Our analysis is based on studying the expected vertex potentials in a random plane graph.
The potential of a vertex is its degree plus the number of vertices visible from it. 
We show that this quantity can be used to study numbers of plane graphs. 
\end{abstract}

\section{Introduction}

In many real-world applications of planar graphs, the vertices are fixed points that cannot be moved. 
This has led computer scientists to study algorithms for graphs where the vertices are fixed points and the edges are non-crossing line segments.
See Figure \ref{fi:PlaneGraphs}.
We refer to such graphs as \emph{plane graphs}, and they are sometimes also called \emph{plane geometric graphs} and \emph{plane rectilinear graphs}.

\begin{figure}[h]
\centering
\includegraphics[scale=0.15]{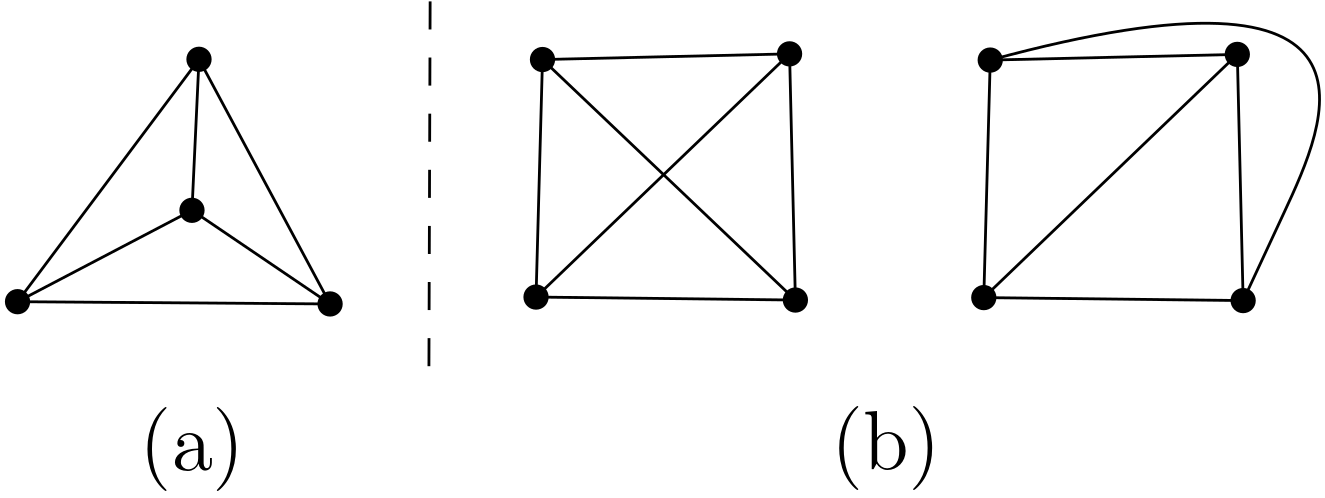}
\caption{ (a) A plane graph. (b) Not plane graphs.}
\label{fi:PlaneGraphs}
\end{figure}

Plane graphs have been significantly studied from an algorithmic perspective. 
For many examples, see \cite{HandbookSU,HandbookTOG}.
However, less is known about the combinatorial aspects of plane graphs.
In particular, they are not as well understood as planar graphs. 

The number of plane graphs that a set of $n$ points may have attracted significant interest. 
After a long sequence of improvements (for example, see \cite{AHHHKV07,ACNS82,GNT00,RSW08}), it is currently known that every set of $n$ points in $\RR^2$ has $O(187.53^n)$ plane graphs \cite{SS13}.
In the other direction, we are only aware of a point set with $\Omega(42.11^n)$ plane graphs \cite{HPS19}, leaving a rather large gap.  

Throughout this work, we always assume that no three points are collinear.
We also assume that the graphs are labeled. 
That is, when considering the graphs that can be drawn over a point set, isomorphic graphs count as distinct.
These are the standard assumptions, made by all works on numbers of plane graphs. 

\begin{figure}[h]
\centering
\includegraphics[scale=0.055]{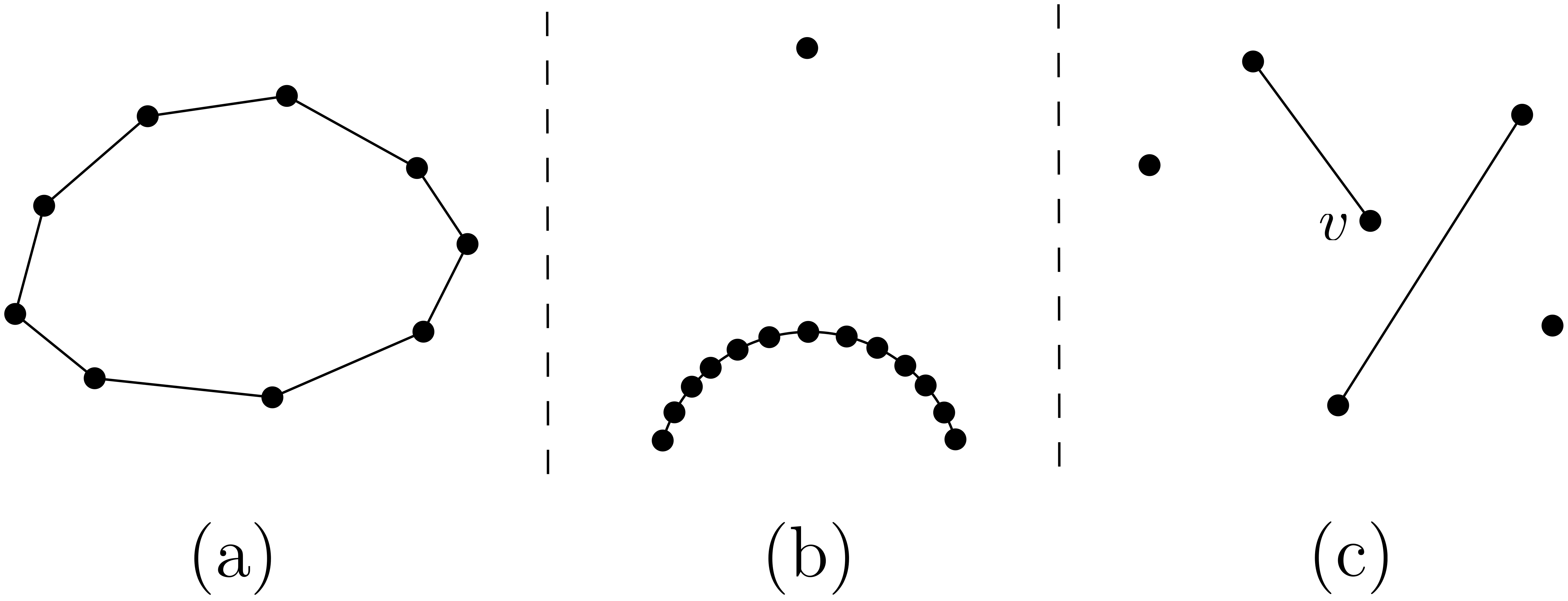}
\caption{(a) Points in convex position. (b) A configuration with a triangular convex hull and approximately $23.31^n$ plane graphs. (c) The ving of $v$ has potential 4.}
\label{fi:Convex}
\end{figure}

The minimum number of plane graphs that can be drawn over a set of $n$ points has also been carefully studied. 
It is known that every set of $n$ points in $\RR^2$ spans $\Omega(11.65^n)$ plane graphs \cite{AHHHKV07}.
This number of plane graphs is achieved by $n$ points in convex position.\footnote{Section \ref{sec:prelim}  provides rigorous definitions for concepts such as convex position and the size of a convex hull.}
See Figure \ref{fi:Convex}(a).
A recent work \cite{LMES24} suggests that point sets 
with a constant-sized convex hull should have significantly more plane graphs. 
Currently, all known configurations with a constant-sized convex hull have $\Omega(23.31^n)$ plane graphs.
This lower bound is obtained by the configuration in Figure \ref{fi:Convex}(b). 
For additional details see \cite{LMES24}.

The first contribution of this work is the first non-trivial bound for the above problem.

\begin{theorem} \label{th:PlaneLower}
Let $\pts$ be a set of $n$ points with a convex hull of size $O(n/\log n)$. 
Then the number of plane graphs that can be drawn on $\pts$ is $\displaystyle \Omega(12.24^n)$.
\end{theorem}

The restriction of a convex hull to be of size $O(n/\log n)$ is somewhat arbitrary. 
Our proof works with any constructive upper bound that is asymptotically smaller than $n$. 

As pointed out in \cite{LMES24}, the lower bound on the minimum number of plane graphs is closely related to the behavior of a random plane graph. 
While random graph theory is a large and established field, almost nothing is known for random plane graphs.
The traditional techniques for studying random graphs do not hold when asking for edges not to intersect. 

In random graph theory, one usually takes each potential edge with a fixed probability.
This approach is no longer possible, since it may lead to intersections, which are not allowed in plane graphs. 
Instead, we uniformly choose a graph from the set of all plane graphs drawn over a given point set.
For the rigorous details, see Section \ref{sec:prelim}.

The aforementioned work \cite{LMES24} studies properties of random plane graphs. 
It proves that, for any set of $n$ points in $\RR^2$ with a triangular convex hull, the expected number of isolated vertices in a random plane graph is smaller than $n/10.18$. 
We derive the following improved bound.

\begin{theorem} \label{th:v0}
Let $\pts$ be a set of $n$ points with a convex hull of size\footnote{The $o(X)$ notation means ``asymptotically smaller than $X$.''} $o(n)$.
Then the expected number of isolated vertices in a uniformly chosen plane graph of $\pts$ is less than $\frac{797}{9760}n < n/12.24$.
\end{theorem}

In the other direction, we only know of a set of $n$ points where the expected number of isolated vertices is larger than $n/23.32$ (see \cite{LMES24}).
Thus, this problem is still far from solved.

So far, the most effective method for studying properties of random plane graphs has been \emph{cross-graph charging schemes}.
In this method, certain vertices are given a charge.
Then, the charge is redistributed among the vertices.
By comparing the total charge before and after the redistribution, we obtain information about properties of a random plane graph. 

Charging schemes are not unusual in graph theory proofs. 
However, in this case, the charge is moved between vertices of different graphs. 
For past works which rely on this approach, see for example \cite{SS13,SS11,SSW11}.

The novelty of our proof is in studying the expected value of a new quantity. 
Let $v$ be a vertex in a plane graph $G$.
The \emph{potential} of $v$ is the degree of $v$ plus the number of vertices that are visible from $v$.
For an example, see Figure \ref{fi:Convex}(c).
For a rigorous definition and further discussion, see Section \ref{sec:prelim}.

The above results are obtained by studying the expected number of vertices of potential 3 and 4 in a random plane graph of a point set.

\begin{theorem} \label{th:potential}
Let $\pts$ be a set of $n$ points with a convex hull of size $o(n)$. Then, in a random plane graph of $\pts$, \\[2mm]
(a) The expected number of vertices of potential 3 is at most $2n/5$, \\[2mm]
(b) The expected number of vertices of potential 3 plus half the expected number of vertices of potential 4 is less than $37n/61$. 
\end{theorem}

Part (a) of Theorem \ref{th:potential} is obtained by a relatively simple and elegant cross-graph charging scheme. 
The proof of part (b) requires a more involved technical analysis. 
The bounds of Theorem \ref{th:potential} are the main new ingredients in proving Theorem \ref{th:PlaneLower} and Theorem \ref{th:v0}.

We believe that studying the expected potentials in random plane graphs may be useful in additional ways. 
In particular, such a study may also help understanding the maximum number of plane graphs that a point set may have. 
We intend to continue exploring this property.

In Section \ref{sec:prelim}, we study the basics of plane graphs more carefully, including random plane graphs and the concept of potential.
In Section \ref{sec:mainTh}, we prove Theorem \ref{th:PlaneLower} and Theorem \ref{th:v0} by relying on Theorem \ref{th:potential}.
In Section \ref{sec:v3}, we prove part (a) of Theorem \ref{th:potential}.
Finally, in Section \ref{sec:v34}, we prove part (b) of Theorem \ref{th:potential}.

\section{Plane graph preliminaries} \label{sec:prelim}

We now study basics of planar point sets and plane graphs.
The reader may wish to first skim this section, and return to it later as needed.
Throughout this section, $\pts$ is a set of $n$ points in $\RR^2$.

\begin{figure}[h]
\centering
\includegraphics[scale=0.25]{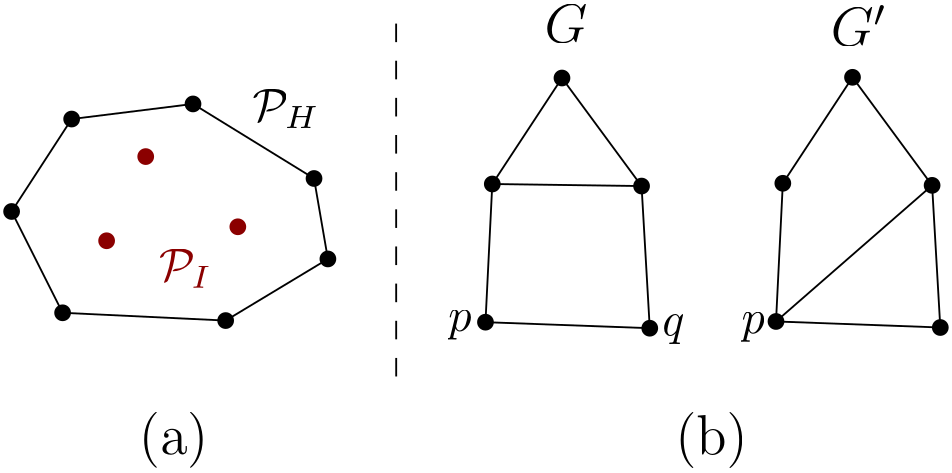}
\caption{(a) Hull points and interior points. (b) Each of $(p,G)$, $(p,G')$, and $(q,G)$ is a distinct ving.}
\label{fi:GraphPrelim}
\end{figure}

The \emph{convex hull} of a point set $\pts$ is the smallest convex polygon that contains $\pts$.
Every convex shape that contains $\pts$ also contains its convex hull. 
See Figure \ref{fi:GraphPrelim}(a).
The vertices of this polygon are points of $\pts$.
We refer to these as the \emph{hull points} of $\pts$ and let $\pts_H$ be 
the set of hull points of $\pts$.
We also set $\pts_I = \pts \setminus \pts_H$ and refer to the elements of $\pts_I$ as the \emph{interior points} of $\pts$. 
See Figure \ref{fi:GraphPrelim}(a) again.

In a \emph{plane graph} of $\pts$, there is a vertex for each point of $\pts$.
The edges are non-crossing line segments.
See Figure \ref{fi:PlaneGraphs}.
Let $\gr(\pts)$ be the set of plane graphs that can be drawn on $\pts$. 
A \emph{ving} (short for \emph{Vertex IN Graph}) is a pair $(p,G) \in \pts_I \times \gr(\pts)$. 
For an example, see Figure \ref{fi:GraphPrelim}(b).
One may think of a ving as an instance of a vertex in one specific graph, as opposed to a vertex that appears in all graphs of $\gr(\pts)$.
We note that vings are defined only with interior points of $\pts$. 
The \emph{degree} of a ving $(p,G)$ is the degree of the vertex $p$ in the graph $G$. 

Let $\pg(\pts)$ be the number of plane graphs that can be drawn over $\pts$.
In other words, $\pg(\pts) = |\gr(\pts)|$.
We also define $\pgk(n) = \max_{\pts}\pg(n)$, where the maximum is over all sets $\pts$ of $n$ points and a convex hull of size $k$. 
In other words, $\pgk(n)$ is the maximum number of plane graphs that a set of $n$ points with a convex hull of size $k$ may have. 
For example, it is not difficult to check that $\pg^3(3)=2^3=8$, that $\pg^3(4)=2^6=64$, and that $\pg^4(4)=2^4\cdot 3=48$.

\parag{Potential.} 
Consider a plane graph $G$ of $\pts$ and points $p,q\in \pts$.
We say that $q$ \emph{is visible from} $p$ in $G$ if $(p,q)$ is not an edge of $G$ and does not intersect any edge of $G$. 
In Figure \ref{fi:GraphPrelim}(b), the top vertex is visible from $p$ in $G'$ but not in $G$. 
The visibility of the ving $(p,G)$ is the number of vertices visible from $p$ in $G$.
In Figure \ref{fi:GraphPrelim}(b), both $(p,G)$ and $(p,G')$ have visibility 1. 

The \emph{potential} of the ving $v=(p,G)$ is the degree of $v$ plus the visibility of $v$.
In other words, the potential of $v$ is the number of vertices to which it is connected or can be connected.
In Figure \ref{fi:GraphPrelim}(b), the ving $(p,G)$ has potential 3, while $(p,G')$ has potential 4.
We denote the potential of $v$ as $\pot(v)$.
We also say that the vertices of $G$ that contribute to $\pot (v)$ \emph{affect} $v$.
That is, the vertices that affect $v$ are the ones that are connected to $v$ in $G$ or visible from $v$.

\begin{figure}[h]
\centering
\includegraphics[scale=0.06]{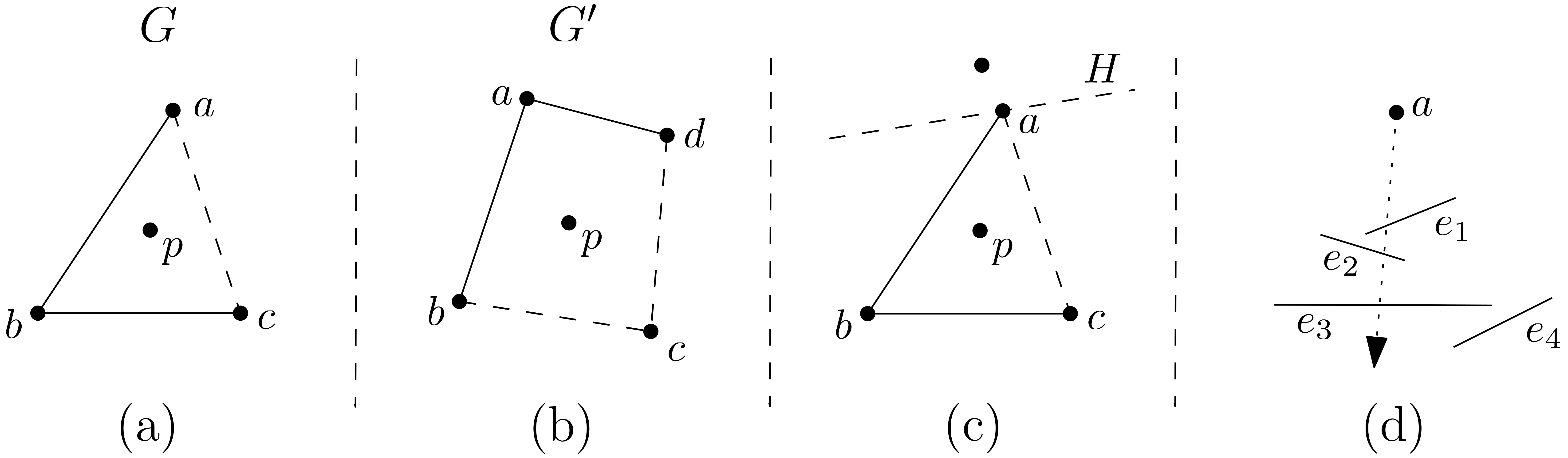}
\caption{ (a) The dashed edge is not in $G$ but is part of the triangle of $(p,G)$. (b) The quadrilateral of $(p,G')$. (c) The closed upper half-plane $H$ intersects $\triangle abc$ in $a$. (d) The ray from $a$ implies that $e_1<e_2<e_3$. }
\label{fi:polygon}
\end{figure}

Consider a potential 3 ving $v=(p, G)$.
Let $a,b,c\in\pts$ be the three vertices that affect $v$. 
We say that $\triangle abc$ is the \emph{triangle of} $v$. 
By definition, this triangle does not contain additional points of $\pts$ and is not intersected by edges that are not sides of the triangle. 
See Figure \ref{fi:polygon}(a).
Similarly, the \emph{quadrilateral} of a potential 4 ving $u$ is the quadrilateral whose four vertices are the ones who affect $u$. 
We imagine shooting an infinite ray from $p$ in an arbitrary direction and then rotating this ray clockwise.
The edges of the quadrilateral of $v$ correspond to pairs of vertices that were visited by the ray consecutively.
We note that this quadrilateral may be either convex or concave.
It does not contain other vertices of $\pts$ and does not intersect other edges.
See Figure \ref{fi:polygon}(b).

\begin{lemma} \label{le:Affected3} 
A potential 3 ving cannot be affected by another potential 3 ving.
\end{lemma}
\begin{proof}
Consider a potential 3 ving $v=(p, G)\in \pts_I \times \gr(\pts)$ that is affected by $(a,G),(b,G)$, and $(c,G)$. 
That is, $(a,G),(b,G),(c,G)$ are the vertices of the triangle of $v$.
To prove the lemma, we show that $u=(a,G)$ cannot be a potential 3 ving. 
First, if $a$ is a hull point of $\pts$, then it does not correspond to a ving. 
We may thus assume that $a$ is not a hull point. 
By definition, each of $(p,G),(b,G),(c,G)$ affects $u$.

We note that there exists a closed half-plane $H$ that intersects $\triangle abc$ only at $a$.
See Figure \ref{fi:polygon}(c). 
We claim that, since $a$ is not a hull vertex, the half-space $H$ contains at least one vertex that affects $a$.
This implies that $\pot(u)\ge 4$ and completes the proof. 

It remains to prove that $H$ contains at least one vertex that affects $a$.
Let $E$ be the set of edges of $G$ that have a non-empty intersection with $H$. 
We define a partial order on $E$, as follows. 
An edge $e_1$ is smaller than an edge $e_2$ if we can shoot a straight ray from $a$ that first hits $e_1$ and then hits $e_2$.
See Figure \ref{fi:polygon}(d).
This partial order is well-defined.

Let $e$ be an edge of $E$ such that no other edge of $E$ is smaller than $e$.
At least one vertex of $e$ is contained in $H$ and this vertex is visible from $a$, as required.
(If $E=\emptyset$, then all vertices of $\pts\cap H$ are visible from $a$. Since $a$ is not a hull vertex, this intersection is non-empty.) 
\end{proof}

Let \emph{hings} be hull points in graphs. 
That is, a hing is a pair $(p,G) \in \pts_H \times \gr(\pts)$. 
Our analysis revolves around vings, but hings are also mentioned a few times.

\begin{figure}[h]
\centering
\includegraphics[scale=0.24]{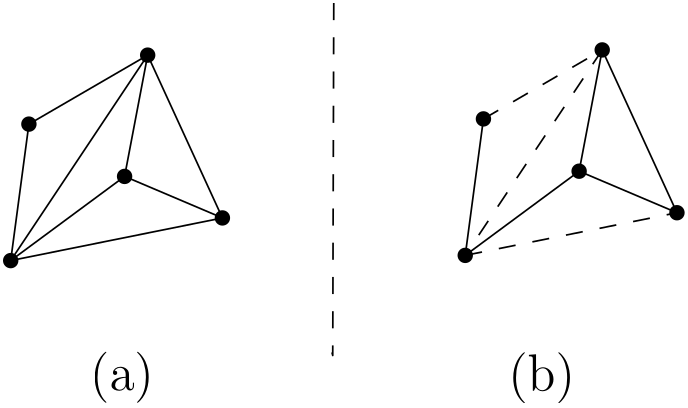}
\caption{ (a) A triangulation of five points. (b) To obtain a triangulation that contains the solid edges, we repeatedly add edges until this is no longer possible. }
\label{fi:triangulation}
\end{figure}

\parag{Triangulations and reflex angles.}
A \emph{triangulation} of $\pts$ is a plane graph of $\pts$ that contains all the edges of the convex hull and where all bounded faces are triangles. 
See Figure \ref{fi:triangulation}(a).
Equivalently, a triangulation of $\pts$ is a plane graph of $\pts$ that is maximal, in the sense that no additional edges can be added.

For each graph $G\in \gr(\pts)$, there exists a triangulation of $\pts$ that contains $G$. 
Indeed, such a triangulation is obtained by repeatedly adding edges to $G$ until this is no longer possible.
See Figure \ref{fi:triangulation}(b).
For a graph $G$, we arbitrarily fix a triangulation $T$ that contains $G$ and denote it as \emph{the triangulation associated with $G$}. 
We also write $T=T(G)$.

\begin{figure}[h]
\centering
\includegraphics[scale=0.34]{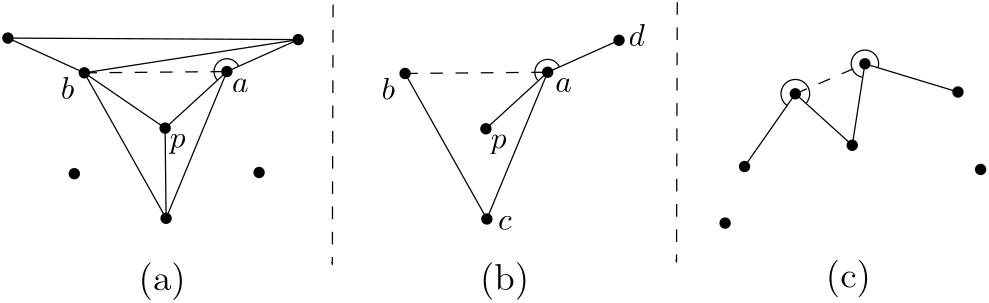}
\caption{ (a) The edge $(a,b)$ forms a reflex angle around $a$. (b) The edges $(a,p)$ and $(a,d)$ form the boundaries of the reflex angle around $(a,b)$. (c) An edge of the convex hull of $\pts$ forms a reflex angle around both of its endpoints. }
\label{fi:reflex}
\end{figure}

Consider a ving or hing $v=(a,G) \in \pts\times \gr(\pts)$.
We say that the edge $(a,b)$ of $T(G)$ \emph{forms a reflex angle around} $v$ if, when removing $(a,b)$ from $T(G)$, the angle that forms around it is larger than $\pi$. 
See Figure \ref{fi:reflex}(a).
Note that $(a,b)$ is not required to exist in $G$. 

\begin{lemma} \label{le:ReflexGen}
Consider a ving or hing $v=(a,G)$.
At most two edges form a reflex angle around $v$.
If two edges form reflex angles, then these are consecutive edges in the order around $a$ in $T(G)$.
\end{lemma}  
\begin{proof}
Assume that the edge $(a,b)$ forms a reflex angle around $v$. 
Let $(a,p)$ and $(a,d)$ be the edges that form the boundaries of the angle around $(a,b)$.
For example, see Figure \ref{fi:reflex}(b).
By definition, the angle between $(a,p)$ and $(a,d)$ in the side that does not contain $(a,b)$ is smaller than $\pi$.
This implies that the only other edges whose removal might lead to an angle larger than $\pi$ are $(a,p)$ and $(a,d)$. 
Since these two angles are disjoint around $a$, at most one of these is larger than $\pi$.
\end{proof}

We note that every edge of the convex hull of $\pts$ forms a reflex angle around both of its endpoints.
For example, see Figure \ref{fi:reflex}(c).

\parag{Random plane graphs.}
We define a \emph{random plane graph} of $\pts$ to be a graph uniformly chosen from $\gr(\pts)$.
That is, each graph of $\gr(\pts)$ is chosen with a probability of $1/\pg(\pts)$.

Let $\hv_0$ be the expected number of vings of degree 0 in a random plane graph of $\pts$.
Let $v_0(G)$ be the number of internal vertices of degree 0 in $G$. 
Then we have that
\[ \hv_0(\pts) = \frac{\sum_{G\in \gr(\pts)}v_0(G)}{\pg(\pts)}.\]
We recall that vings correspond only to internal points of $\pts$.
For example, if $\pts$ is a set of points in convex position, then $\hv_0(\pts) =0$.

For a graph $G\in \gr(\pts)$, let $p_i(G)$ be the number of potential $i$ vings in $G$. 
Let $\hp_i(\pts)$ denote the expected size of $p_i(G)$.
In other words, 
\[ \hp_i(\pts) = \frac{\sum_{G\in \gr(\pts)}p_i(G)}{\pg(\pts)}.\]

\section{Proofs of Theorems \ref{th:PlaneLower} and \ref{th:v0}} \label{sec:mainTh}

In this section, we prove Theorems \ref{th:PlaneLower} and \ref{th:v0} by relying on Theorem \ref{th:potential}.
The proof of this latter result appears in Sections \ref{sec:v3} and \ref{sec:v34}.

As a first step, we derive a variant of a standard result. 
For example, see \cite[Lemma 1.4]{LMES24}.

\begin{lemma} \label{le:nVSn-1}
Consider fixed $3\le k < n$. 
If $\delta > 0$ satisfies $\hv_0 (\pts) \le \delta n$ for every set $\pts$ of $n$ points with a convex hull of size $k$, 
then $\displaystyle \pgk(n) \ge \frac{n-k} {\delta n} \cdot  \pgk(n-1)$.
\end{lemma}
\begin{proof}
Let $\pts$ be a set that minimizes $\pg(\pts)$
among all sets of $n$ points and a convex hull of size $k$.
We can get some plane graphs of $\pts$ by
choosing an internal point $q \in \pts_I$ and
a plane graph $H$ of $\pts \setminus \{q\}$, and then inserting $q$ to $H$ with no edges.
A plane graph $G$ of $\pts$ can be obtained
in exactly $v_0(G)$ ways in this manner.
This implies that
\[
\hv_0(\pts)\cdot \pg(\pts) = \sum_{G\in \gr(\pts) } v_0(G) = \sum_{q \in \pts_I} \pg(\pts \setminus\{q\}).
\]

The left part of the above equation equals $\hv_0(\pts)\cdot \pgk(n)$.
The right-hand part is at least $(n-k)\cdot \pgk(n-1)$. 
Recalling that $\hv_0(\pts) \le \delta n$ leads to
\[ \pgk(n) = \pg(\pts) \ge \frac{n-k}{\hv_0(\pts)} \cdot \pgk(n-1) \ge
\frac{n-k}{\delta n} \cdot \pgk(n-1). \]
\end{proof}

We are now ready to prove Theorem \ref{th:v0}.
We first recall the statement of this result.
\vspace{2mm}

\noindent {\bf Theorem \ref{th:v0}.}
\emph{Let $\pts$ be a set of $n$ points with a convex hull of size $o(n)$. 
Then} $\displaystyle \hv_0(\pts) < 797n/9760$.

\begin{proof}
We start by giving one unit of charge to every isolated ving $(p,G)\in \pts_I\times \gr(\pts)$. 
Vings with positive degrees have no charge.
Let $t$ be the total charge summed over all vings in $\pts_I\times \gr(\pts)$.
Then we may write 
\begin{equation} \label{eq:t}
t = \sum_{G \in \gr(\pts)} v_0(G) = \hv_0(\pts) \cdot pg(\pts). 
\end{equation}

\begin{figure}[h]
\centering
\includegraphics[scale=0.12]{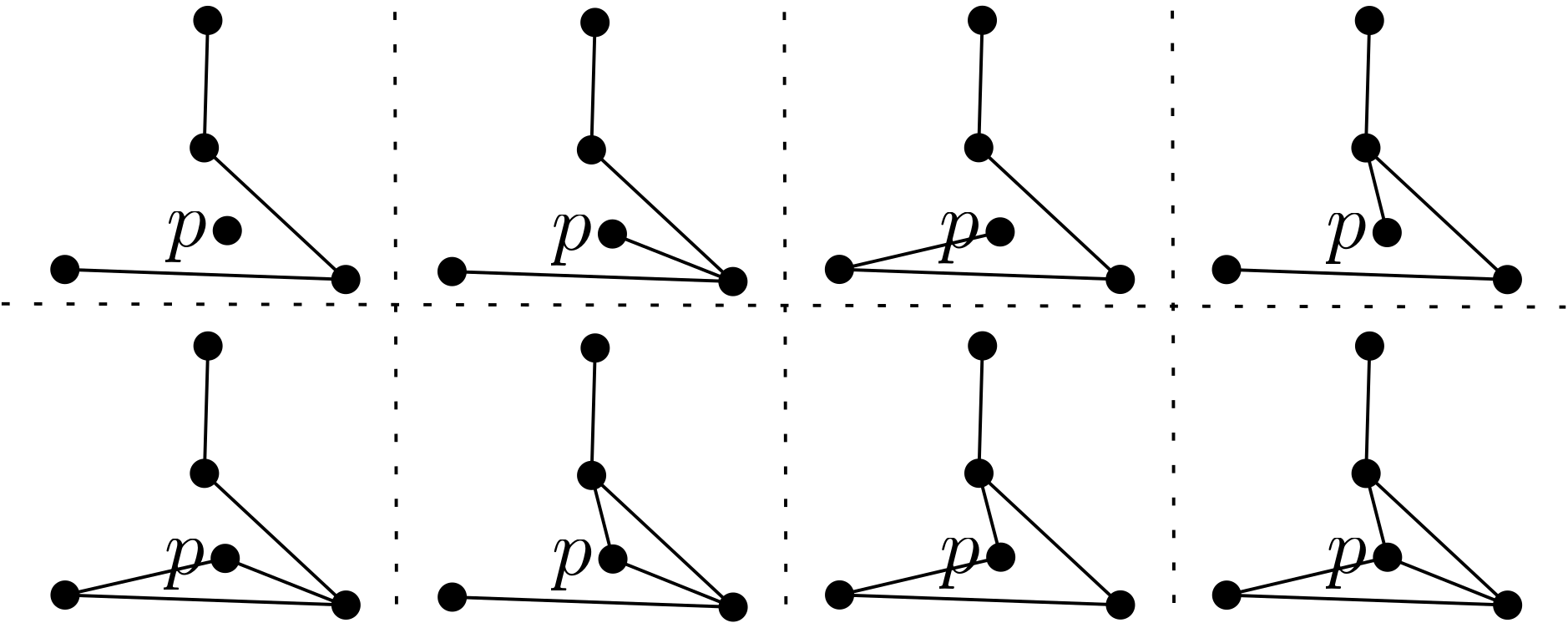}
\caption{ The family of an isolated ving $(p,G)$ of potential 3 has eight members. }
\label{fi:family}
\end{figure}

The \emph{family} of an isolated ving $v=(p,G)\in \pts_I \times G(\pts)$ is the set of all vings obtained by starting with $v$ and connecting $p$ to any number of vertices visible from $p$.
For example, see Figure \ref{fi:family}.
We note that, if $v$ has visibility $i$, then the family consists of $2^i$ vings, all with the same point of $\pts_I$.
We say that \emph{the visibility of the family} is $i$.
We claim that each ving $(q,H)$ belongs to exactly one family. 
Indeed, $(q,H)$ belongs to the family of the isolated ving obtained by removing from $H$ all edges that have $q$ as an endpoint.

We redistribute the charge on the vings of $\pts_I\times G(\pts)$, as follows.
For each family, we take the unit charge of the isolated ving of the family and equally split it among all members of that family. 
That is, a ving in a family of visibility $i$ receives a charge of $1/2^i$.

The visibility of the family of a ving $v$ is $\pot (v)$.
This implies that the charge of $v$ is $1/2^{\pot (v)}$.
We may thus write
\begin{align*} 
t &= \sum_{v \in \pts_I \times \gr(\pts)} \frac{1}{2^{\pot (v)}} \le \sum_{v \in \pts_I \times \gr(\pts) \atop \pot(v)=3} \frac{1}{8} + \sum_{v \in \pts_I \times \gr(\pts) \atop \pot(v)=4} \frac{1}{16} + \sum_{v \in \pts_I \times \gr(\pts) \atop \pot(v)>4} \frac{1}{32}\\[2mm]
&< \frac{1}{8}\cdot \hp_3(\pts)\cdot \pg(\pts) + \frac{1}{16}\cdot \hp_4(\pts)\cdot \pg(\pts) + \frac{1}{32}\cdot (n-\hp_3(\pts)-\hp_4(\pts))\cdot \pg(\pts).
\end{align*}

Assuming that $\hp_3(\pts)=2n/5$ and combining this with Theorem \ref{th:potential} leads to 
\begin{align*} 
t < \frac{1}{8}\cdot \frac{2n}{5}\cdot \pg(\pts) &+ \frac{1}{16}\cdot 2 \left(\frac{37n}{61}-\frac{2n}{5}\right)\cdot \pg(\pts) + \frac{1}{32}\cdot \left(n-\frac{2n}{5}-2\left(\frac{37n}{61}-\frac{2n}{5}\right)\right)\cdot \pg(\pts)  \\[2mm]
&\quad \hspace{19mm}=\frac{n}{20}\cdot \pg(\pts) + \frac{63n}{2440}\cdot \pg(\pts) + \frac{57n}{9760}\cdot \pg(\pts) = \frac{797n}{9760} \cdot \pg(\pts). 
\end{align*}
Together with \eqref{eq:t}, we get that
\[ \hv_0(\pts) < \frac{797n}{9760}. \]

When removing the assumption $\hp_3(\pts)=2n/5$, we may decrease $\hp_3(\pts)$ and increase $\hp_4(\pts)$.
In particular, if we decrease $\hp_3(\pts)$ by $s$, we may increase $\hp_4(\pts)$ by $2s$ (as long as $\hp_4(\pts)\le n$). 
This increases the upper bound on $t$ by at most $-s/8 + 2s/16 =0$.
Thus, the above upper bound for $t$ still holds, and so does the upper bound on $\hv_0(\pts)$.
\end{proof}

We are now ready to prove Theorem \ref{th:PlaneLower}.
We first recall the statement of this result. 
\vspace{2mm}

\noindent {\bf Theorem \ref{th:PlaneLower}}
\emph{Let $\pts$ be a set of $n$ points with a convex hull of size $O(n/\log n)$. 
Then $\displaystyle \pg(\pts) = \Omega(12.24^n)$.}

\begin{proof}
We denote the size of the convex hull of $\pts$ as $k=O(n/\log n)$.
We also set $m = n - n/\sqrt{\log n}$ and $\delta = 797/9760$.
By applying $m$ times Lemma \ref{le:nVSn-1} and Theorem \ref{th:v0}, we obtain that
\begin{align*}
\pg(\pts) &\ge \pg^k(n) \ge \frac{n-k}{\delta \cdot n}\cdot \frac{n-1-k}{\delta \cdot (n-1)}\cdot \frac{n-2-k}{\delta \cdot (n-2)}\cdots \frac{n/\sqrt{\log n}-k+1}{\delta \cdot (n/\sqrt{\log n}+1)}\cdot \pg^k (n/\sqrt{\log n}) \\[2mm]
&=\frac{1}{\delta^m} \cdot \frac{(\frac{n}{\sqrt{\log n}}-k+1)\cdot (\frac{n}{\sqrt{\log n}}-k+2)\cdot (\frac{n}{\sqrt{\log n}}-k+3)\cdots(\frac{n}{\sqrt{\log n}})}{n \cdot (n-1)\cdot (n-2) \cdots (n-k+1)} \cdot \pg^k (n/\sqrt{\log n}).
\end{align*}

By assuming that $n$ is sufficiently large and noting that $\pg^k (n/\sqrt{\log n})>1$, we get that
\begin{align*}
\pg(\pts) &>\frac{1}{\delta^m} \cdot \frac{(n/2\sqrt{\log n})^k}{n^k} = \left(\frac{1}{\delta}\right)^n \cdot \frac{\delta^{n/\sqrt{\log n}}}{(2\sqrt{\log n})^k} \\[2mm]
&=\left(\frac{9760}{797}\right)^n \cdot \frac{\delta^{n/\sqrt{\log n}}}{(2\sqrt{\log n})^k} = \Omega\left(12.24^n\right).
\end{align*}
\end{proof}

\section{An upper bound for $\hp_3(\pts)$} \label{sec:v3}

In this section, we prove the first part of Theorem \ref{th:potential}.
The proof is inspired by \cite[Lemma 2.1]{SSW11}.
\vspace{2mm}

\noindent {\bf Theorem \ref{th:potential}(a).}
\emph{Let $\pts$ be a set of $n$ points with a convex hull of size $o(n)$. 
Then} $\displaystyle \hp_3(\pts) \le 2n/5$.

\begin{proof}
We prove the theorem using a cross-graph charging scheme. 
First, we give a charge of 3 to every potential 3 ving in $\pts_I\times \gr(\pts)$.
Denoting the total charge as $t$, we have that
\begin{equation} \label{eq:TotalCharge} 
t = 3\sum_{G\in \gr(\pts)} p_3(G) = 3\cdot \hp_3(\pts)\cdot \pg(\pts).
\end{equation}

Let \emph{high-potential vings} be the vings of potential at least 4.
Below, we show how to move the entire charge from potential 3 vings to high-potential vings and hings, so that every high-potential ving and hing is charged at most 2. 
This implies that 
\[ t\le 2\sum_{G\in \gr(\pts)} \left(n-p_3(G)\right) = 2\cdot \pg(\pts) \cdot \left(n-\hp_3(\pts)\right). \]

Combining the above with \eqref{eq:TotalCharge} leads to 
\[ \hp_3(\pts) \le 2n/5. \]
It remains to explain how the charge is moved from potential 3 vings.

\parag{Charging rules.}
Consider a potential 3 ving $v=(p,G)\in \pts_I\times \gr(\pts)$. 
As stated above, $v$ starts with a charge of 3.
Then, each unit of charge corresponds to a different edge of the triangle of $v$. 
Let the vertices of the triangle of $v$ be $a,b,c\in\pts$. 
We discuss the charge that corresponds to the edge $(a,b)$, which may not be an edge of $G$.

\begin{figure}[h]
\centering
\includegraphics[scale=0.07]{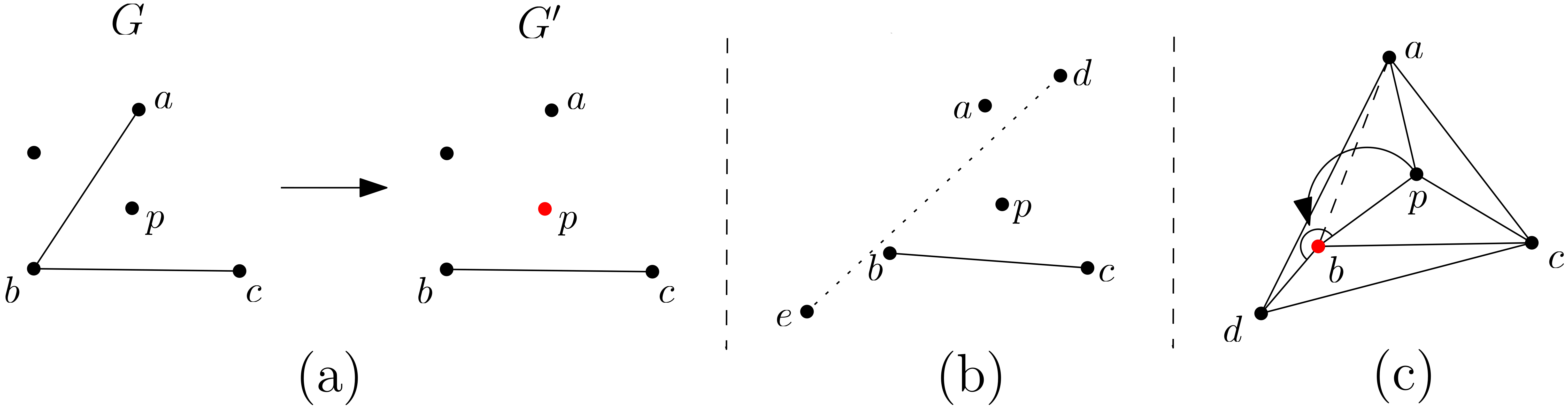}
\caption{ (a) A removal charge from $(p,G)$ to $(p,G')$, corresponding to the edge $(a,b)$. (b) If an edge $(d,e)$ in $T(G)$ intersects $(a,b)$, then at least one of $d$ and $e$ affects $(p,G)$. (c) The edge $(a,b)$ is part of the two triangles $\triangle abp$ and $\triangle abd$ in $T(G)$. }
\label{fi:ChargingRules}
\end{figure}

We first consider the case where the edge $(a,b)$ is in $G$ and removing $(a,b)$ increases the potential of $p$. 
That is, removing $(a,b)$ increases the number of vertices visible from $p$. 
Let $G'$ be the graph obtained by removing $(a,b)$ from $G$. 
We set $v'=(p,G')$ and note that $v'$ is a high-potential ving. 
We move the corresponding unit of charge from $v$ to $v'$.
For an example, see Figure \ref{fi:ChargingRules}(a).
We refer to such a charge transfer as a \emph{removal charge}.

We now assume that we are not in the above case. 
That is, either $(a,b)$ is not in $G$ or removing $(a,b)$ from $G$ does not increase the number of vertices that are visible from $p$.

\begin{lemma} \label{le:NoRemReflex}
If $(a,b)$ does not correspond to a removal charge, then $(a,b)$ forms a reflex angle with at least one of its endpoints.
\end{lemma}
\begin{proof}
First, we claim that $(a,b)$ must exist in $T(G)$.
Indeed, for $(a,b)$ to not be in $T(G)$, this triangulation must contain an edge that intersects $(a,b)$.
This edge cannot be of the form $(p,d)$ for some $d\in \pts$, since $v$ is only affected by $(a,G),(b,G),(c,G)$. 
Similarly, this edge cannot be an edge of the form $(c,d)$, since this also implies that $(d,G)$ affects $v$.
This leaves the case where the edge is of the form $(d,e)$, for another $e\in\pts$, as depicted in Figure \ref{fi:ChargingRules}(b).
This is impossible, since at least one of $d$ and $e$ will affect $p$. 

If $(a,b)$ is an edge of the convex hull of $\pts$, then it forms a reflex angle around both of its endpoints.
It thus remains to consider the case where $(a,b)$ is not part of the convex hull.
In this case, $(a,b)$ is part of two triangles $\triangle abp$ and $\triangle abd$ in $T(G)$.
See Figure \ref{fi:ChargingRules}(c).
By definition, $d$ is not visible from $p$, also after removing $(a,b)$.
This is equivalent to one of $\angle pad$ and $\angle pbd$ being a reflex angle.
\end{proof}

If $(a,b)$ does not correspond to a removal charge, then it forms a reflex angle with one of its endpoints. 
We move the corresponding unit of charge to a vertex around which $(a,b)$ forms a reflex angle.
If $(a,b)$ forms a reflex angle around both of its endpoints, then we choose one endpoint arbitrarily.
We refer to such a charge transfer as a \emph{reflex charge}.
An edge of the convex hull of $\pts$ always leads to a reflex charge.

\parag{The charge of high-potential vings and hings.}
We now analyze the maximum charge that a high-potential ving or hing may have after the above charge transfer.
We begin by studying removal charges more carefully.

\begin{figure}[h]
\centering
\includegraphics[scale=0.073]{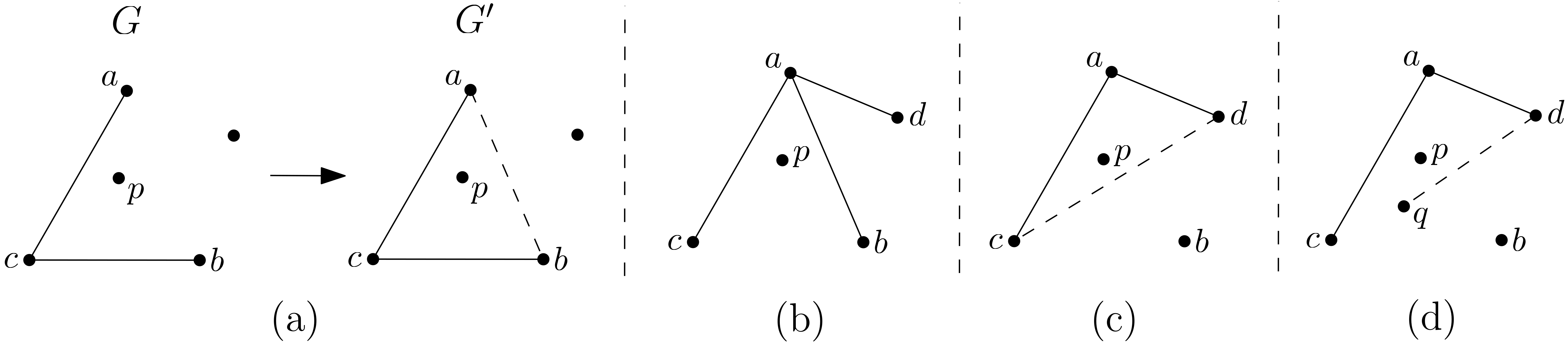}
\caption{(a) A removal charge corresponds to an edge addition that decreases the potential of $p$ to 3. 
(b) The edge $(a,b)$ blocks the visibility between $p$ and $d$. 
(c) The edge $(c,d)$ blocks the visibility between $p$ and $b$. 
(d) If $(d,q)$ blocks the visibility between $p$ and $b$, then $q$ affects $p$.}
\label{fi:removal3}
\end{figure}

\begin{lemma} \label{le:removal3}
Let $v=(p,G)$ be a high-potential ving.
Then $v$ receives at most two removal charges. Additionally, if $v$ receives two removal charges then it does not receive reflex charges.
\end{lemma}
\begin{proof}
Consider a potential 3 ving $v'=(p,G')$ that makes a removal charge to $v$. Then $G'$ is obtained from $G$ by adding a single edge of the triangle of $v'$.
See Figure \ref{fi:removal3}(a).
Denote the vertices of the triangle of $v'$ as $a,b,c\in \pts$ and let the removed edge be $(a,b)$.

Consider a second potential 3 ving $v''=(p,G'')$ that makes a removal charge to $v$.
Since the triangles of $v'$ and $v''$ cannot be identical, the triangle of $v''$ involves at least one new vertex $d\in \pts$. 
For $d$ not to be visible from $v'$, it must be hidden by the additional edge $(a,b)$.
See Figure \ref{fi:removal3}(b).
For $d$ to be in the triangle of $v''$, it must be an endpoint of an edge $e$ that blocks at least one of $a,b,c$. 
We note that the other endpoint of $e$ cannot be $a$ or $b$. 
It is possible that $e=(c,d)$, blocking either $a$ or $b$. 
See Figure \ref{fi:removal3}(c).
Finally, the other endpoint of $e$ cannot be another vertex $q\in \pts$, since then $q$ would have been visible from $v'$, making $v'$ have potential 4. 
See Figure \ref{fi:removal3}(d).

\begin{figure}[h]
\centering
\includegraphics[scale=0.11]{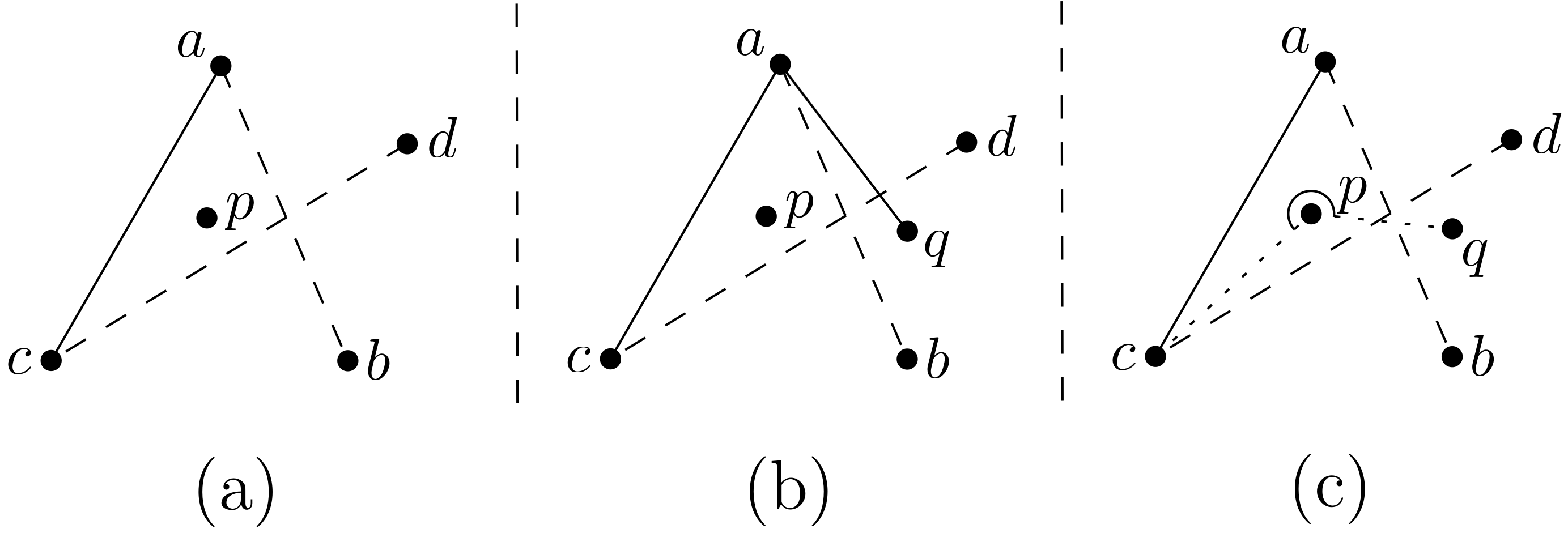}
\caption{ (a) A ving that receives two removal charges. (b) The visibility between $p$ and $q$ is blocked by both $(a,b)$ and $(c,d)$. (c) A reflex angle $\angle cpq$.}
\label{fi:removal3c}
\end{figure}

By the above, the only way for $v$ to receive two removal charges is as depicted in Figure \ref{fi:removal3c}(a).
To have a third removal charge, there should exist a third edge $e'$ whose addition to $G$ decreases the potential of $v$ down to 3. 
This edge must have a new vertex $q\in \pts$ as an endpoint. 
The visibility between $p$ and $q$ must be blocked each of $(a,b)$ and $(c,d)$.
See Figure \ref{fi:removal3c}(b).
Adding $e'$ cannot give $p$ potential 3. 
We conclude that $v$ receives at most two removal charges.

Finally, we assume that $v$ receives two removal charges, as described above. 
For $v$ to also receive a reflex charge, this charge must originate from a potential 3 ving $u$ that is visible from $v$. 
Since $abcd$ is a convex quadrilateral, the potential of each of $a,b,c,d$ is at least 4.
Any other vertex that is visible in $G$ from $p$ is behind both $(a,b)$ and $(c,d)$. 
See Figure \ref{fi:removal3c}(c).
Such a ving also cannot have potential 3. 
Thus, $v$ cannot receive reflex charges.
\end{proof}

By definition, a hing does not receive removal charges.  
Next, we move to study reflex charges.

\begin{figure}[h]
\centering
\includegraphics[scale=0.15]{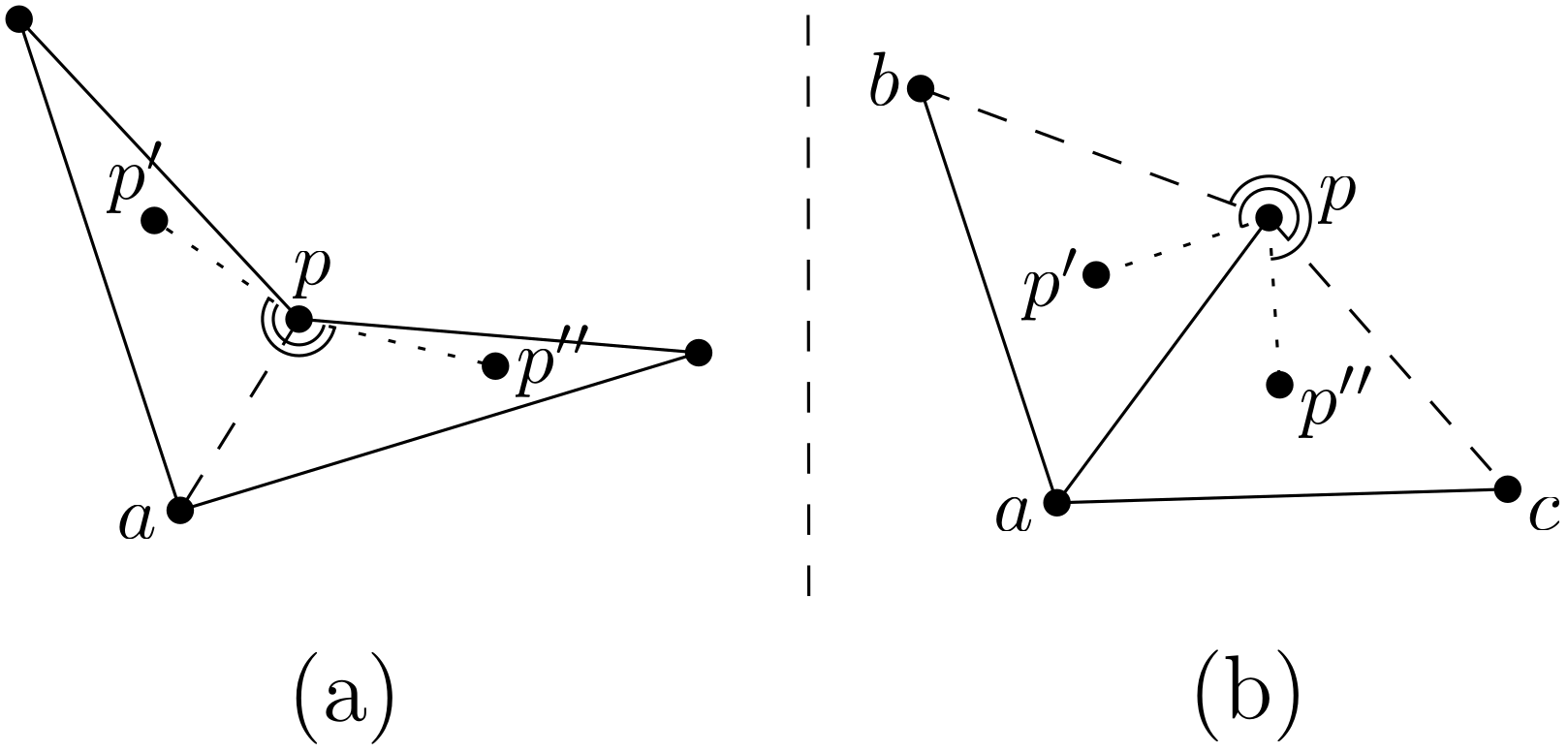}
\caption{(a) The edge $(a,p)$ corresponds to two reflex charges: one from $p'$ and another from $p''$. (b) The ving $p$ receives reflex charges that correspond to the edges $(p,b)$ and $(p,c)$.}
\label{fi:TwoReflex}
\end{figure}

\begin{lemma} \label{le:reflex3}
Let $v=(p,G)$ be a high-potential ving or a hing.
Then $v$ receives at most two reflex charges. 
If $v$ receives two reflex charges then it does not receive removal charges.
\end{lemma}
\begin{proof}
Lemma \ref{le:ReflexGen} states that at most two edges form a reflex angle around $v$.
This does not immediately imply that $v$ receives at most two reflex charges, since the same edge may correspond to two reflex charges.
For example, see Figure \ref{fi:TwoReflex}(a).

Assume that $v$ receives two reflex charges that correspond to the same edge $(a,p)$, as in Figure \ref{fi:TwoReflex}(a).
Let the vertices of the two charging potential 3 vings be $p'$ and $p''$.
By Lemma \ref{le:ReflexGen}, the only other edges that may potentially form a reflex angle around $p$ in $G$ are $(p,p')$ and $(p,p'')$ (these two edges must exist in $T(G)$.)
Since $(p',G)$ and $(p'',G)$ are potential 3 vings, neither edge forms a reflex angle around $p$.

It remains to show that, when $v$ receives two reflex charges, it cannot receive removal charges.
That is, we cannot decrease the potential of $v$ down to 3 by adding one additional edge.
By inspecting Figures \ref{fi:TwoReflex}(a) and \ref{fi:TwoReflex}(b), we see that this indeed must be the case.
\end{proof}

Lemmas \ref{le:removal3} and \ref{le:reflex3} imply that every high potential ving and every hing receive a charge of at most 2.
This completes the proof of the theorem.
\end{proof}

\section{An upper bound for $\hp_3(\pts)+\hp_4(\pts)/2$} \label{sec:v34}

In this section we derive an upper bound on $\hp_3(\pts)+\hp_4(\pts)/2$, when $\pts$ does not have a large convex hull.
This requires a longer and more technical analysis than the one of Section \ref{sec:v3}. 
\vspace{2mm}

\noindent {\bf Theorem \ref{th:potential}(b).} 
\emph{Let $\pts$ be a set of $n$ points with a convex hull of size $o(n)$. Then} $\displaystyle \hp_3(\pts) +\hp_4(\pts)/2< 37n/61$.

\begin{proof}
We prove the theorem by using a cross-graph charging scheme. 
First, we give a charge of 3 to every ving in $\pts\times \gr(\pts)$ of potential 3, and a charge of $3/2$ to every ving of potential 4.
Denoting the total charge as $t$, we have that
\begin{equation} \label{eq:TotalCharge34} 
t = 3\sum_{G\in \gr(\pts)} (p_3(G) + p_4(G)/2) = 3\cdot (\hp_3(\pts)+\hp_4(\pts)/2)\cdot \pg(\pts).
\end{equation}

Let \emph{high-potential vings} be the vings of potential at least 5.
Below, we redistribute the charge, so that every high-potential ving is charged at most $37/8$ and every hing at most $2$. 
Every potential 4 ving is left with a charge of at most $7/4$ and every potential 3 ving has no charge.
The number of high-potential vings and hings in a graph $G\in \gr(\pts)$ is $n-p_3(G)-p_4(G)$. 
We get that
\begin{align*} 
t &< \frac{37}{8}\sum_{G\in \gr(\pts)} \left(n-p_3(G)-p_4(G)\right) + \frac{7}{4}\sum_{G\in \gr(\pts)} p_4(G) \\[2mm]
&\le \pg(\pts) \cdot \left(\frac{37n}{8}-\frac{37}{8}\cdot \hp_3(\pts)-\frac{23}{8}\cdot \hp_4(\pts)\right) \\[2mm]
&< \pg(\pts) \cdot \left(\frac{37n}{8}-\frac{37}{8}\cdot \hp_3(\pts)-\frac{37}{16}\cdot \hp_4(\pts)\right) = \pg(\pts) \cdot \left(\frac{37n}{8}-\frac{37}{8}\left( \hp_3(\pts)+\frac{\hp_4(\pts)}{2}\right)\right).
\end{align*}

Combining the above with \eqref{eq:TotalCharge34} leads to 
\[ \hp_3(\pts) + \frac{\hp_4(\pts)}{2}< \frac{37n}{61}, \]
as asserted by the theorem.
It remains to describe the charge redistribution.

\parag{Charging rules.}
We redistribute the charge in two sequential stages. 
In the first stage, we move charge from potential 3 vings, as described in the proof of Theorem \ref{th:potential}(a).
In Section \ref{sec:v3}, we proved that each ving of potential at least 4 receives at most 2 units from such a redistribution.
Thus, after the first stage of charge redistribution,  the charge of each potential 4 ving is $3/2$, $5/2$, or $7/2$.
Each high potential ving has at most two units of charge.
Potential 3 vings have no charge.

We now study the second stage of redistributing charge. 
Consider a potential 4 ving $v=(p,G)\in \pts_I\times \gr(\pts)$. 
We evenly divide the charge of $v$ among the four edges of the quadrilateral of $v$ (some of these edges may not be in $G$).
That is, each edge corresponds to a charge transfer of at least $3/8$ and at most $7/8$. 
The following charging rules describe how the charge is moved according to the properties of the corresponding edge.

Let the vertices of the quadrilateral of $v$ be $a,b,c,d\in\pts$. 
We discuss the charge that corresponds to the edge $(a,b)$, which may not be an edge of $G$.
We define a \emph{removal charge} as in the proof of Theorem \ref{th:potential}(a):
Assume that the edge $(a,b)$ is in $G$ and let $G'$ be the graph obtained by removing $(a,b)$ from $G$. 
If $(p,G')$ has a higher potential than $v$, then we move the corresponding charge to $(p,G')$. 

\begin{figure}[h]
\centering
\includegraphics[scale=0.102]{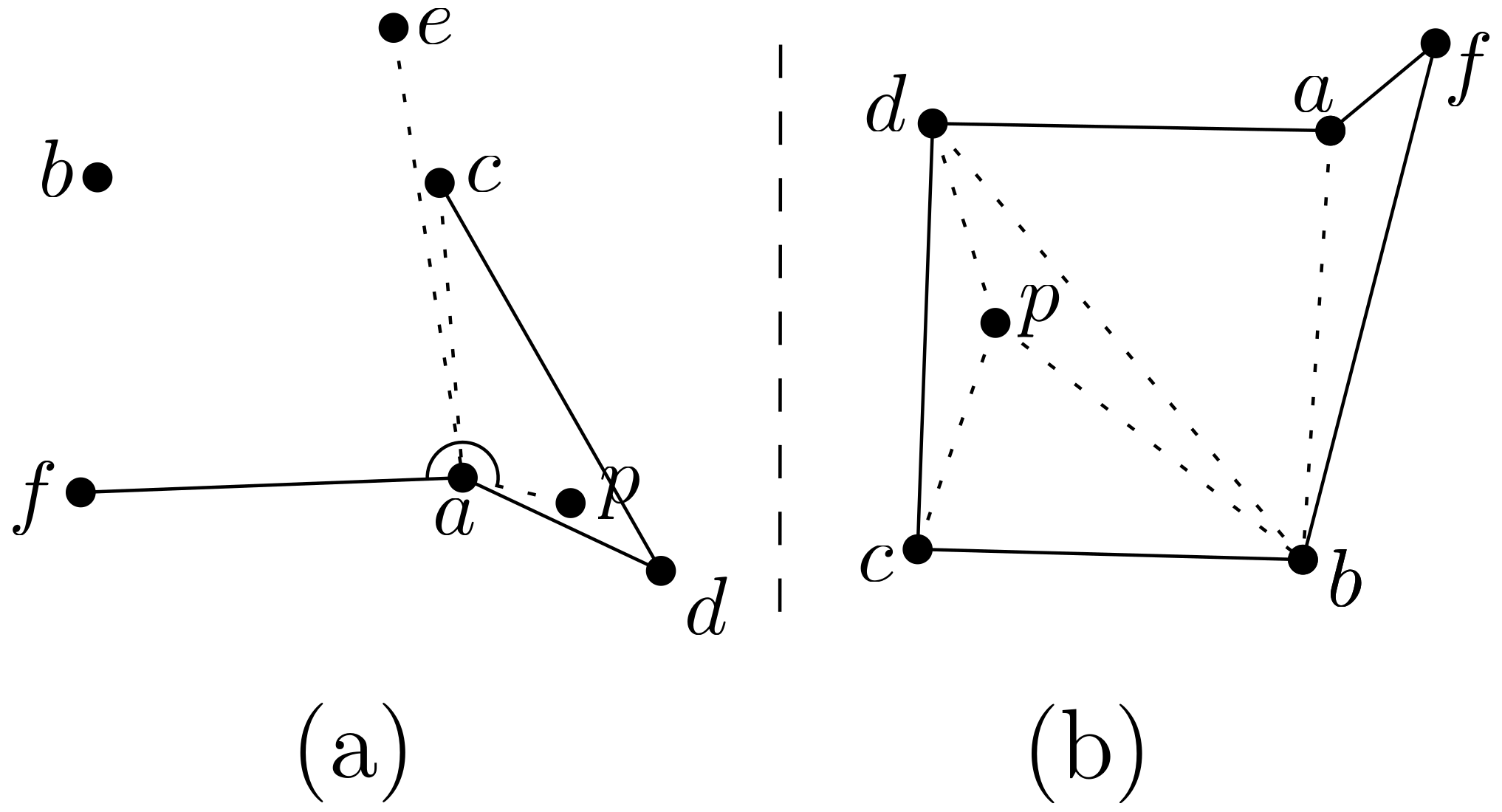}
\caption{Dotted edges are in $T(G)$ but not in $G$. (a) The angle $\angle fap$ is broken by three edges: $(a,c),(a,e),(a,b)$. (b) While we would like the potential 4 ving $(p,G)$ to have a reflex charge corresponding to $(a,b)$, the edge $(p,a)$ is not in $T(G)$. }
\label{fi:Pseudo-Reflex}
\end{figure}

In Section \ref{sec:v3}, Lemma \ref{le:NoRemReflex} states that, when $(a,b)$ does not correspond to a removal charge, this edge forms a reflex angle around $a$ or $b$.
This argument does not extend to potential 4 vings. 
Assume that $(a,b)$ is not a boundary edge. 
Let $\triangle abf$ be the triangle in $T(G)$ that has $(a,b)$ on its boundary and does not intersect the interior of the quadrilateral of $v$.
Since $(a,b)$ does not correspond to a removal charge, either $\angle fap$ or $\angle fbp$ is a reflex angle. 
Without loss of generality, we assume that $\angle fap$ is the reflex angle. 
It is possible that $\angle fap$ is broken by more than a single edge in $T(G)$.
See Figure \ref{fi:Pseudo-Reflex}(a).
In addition, it is possible that $(a,p)$ is not an edge of $T(G)$. 
See Figure \ref{fi:Pseudo-Reflex}(b).

In either of the above cases, we keep the corresponding charge on $(p,G)$.
We refer to this as a \emph{self charge} of \emph{type (i)}.
Lemma \ref{le:NoRemReflex} holds for potential 3 vings, as before. 
We now prove a variant of Lemma \ref{le:NoRemReflex} for potential 4 vings. 

\begin{figure}[h]
\centering
\includegraphics[scale=0.075]{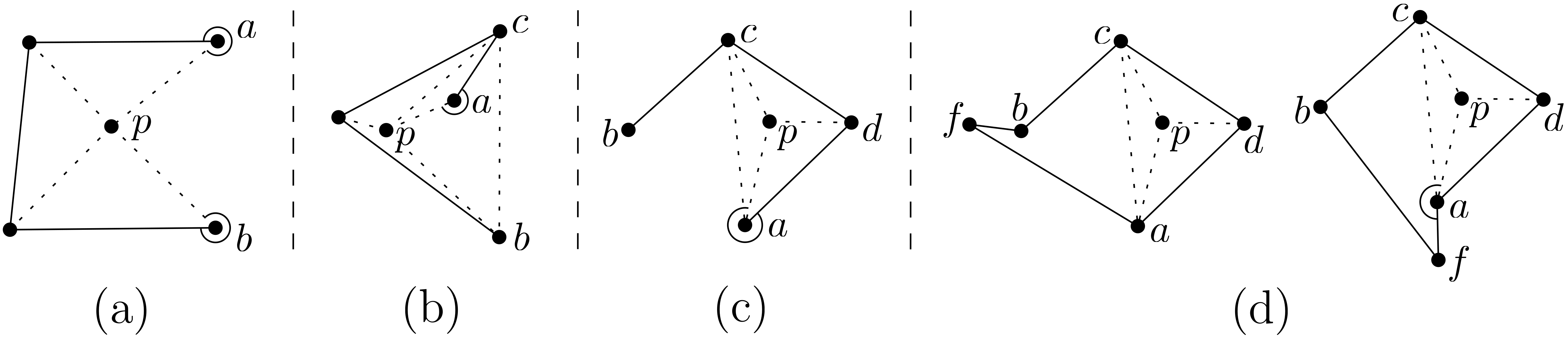}
\caption{Dotted edges are in $T(G)$ but not in $G$. (a) The edges $(p,a)$ and $(p,b)$ are in $T(G)$. (b) The third vertex of the triangle involving $(a,b)$ may be a vertex of the quadrilateral of $v$. (c) The edge $(a,c)$ breaks the reflex angle $\angle pad$, causing type (i) self charge.
 (d) When $(a,c)\in T$ and $(a,b)$ is not a hull edge, we have a type (i) self charge.}
\label{fi:Lemma43a.png}
\end{figure}

\begin{lemma} \label{le:RemRefSelf}
Let $v=(p,G)$ be a potential 4 ving. 
Let $(a,b)$ be an edge of the quadrilateral of $v$. \\[2mm]
(a) If both $(p,a)$ and $(p,b)$ are in $T(G)$, then $(a,b)$ corresponds to a reflex or removal charge. \\[2mm]
(b) Without the assumption of part (a), the edge $(a,b)$ corresponds to a removal charge, a reflex charge, or a type (i) self charge.
\end{lemma}
\begin{proof}
(a) We assume that $(p,a)$ and $(p,b)$ are in $T(G)$, as in Figure \ref{fi:Lemma43a.png}(a).
If $(a,b)$ is an edge of the convex hull of $\pts$, then $v$ makes a reflex charge to $(a,G)$ or $(b,G)$ (arbitrarily chosen). 
We thus assume that $(a,b)$ is not a boundary edge.
Let $c$ be the third vertex of the triangle of $T(G)$ that contains $(a,b)$ but not $p$.
It is possible that $c$ is a vertex of the quadrilateral of $(p,G)$, as depicted in Figure \ref{fi:Lemma43a.png}(b).
If $(a,b)$ does not correspond to a removal charge, then either $\angle pac$ or $\angle pbc$ is a reflex angle.
By definition, $(a,b)$ corresponds to a reflex charge of $v$.

(b) In this case, $T(G)$ contains at most one of $(p,a), (p,b)$. 
In $T(G)$, the vertex $p$ is connected to at least three vertices of the quadrilateral of $v$. 
Thus, exactly one of $(p,a),(p,b)$ exists  in $T(G)$.
Without loss of generality, we assume that $(p,a)$ is in $T(G)$, but not $(p,b)$.
We first assume that $(a,b)$ is part of the convex hull of $\pts$.
Then $(a,c)$ breaks the reflex angle $\angle pab$, causing type (i) self charge.
See Figure \ref{fi:Lemma43a.png}(c).

Finally, we assume that $(a,b)$ is part of a triangle in $T(G)$ whose third vertex is not $c$. 
Let $f$ be the third vertex of that triangle.
If $(a,b)$ does not correspond to a removal charge, then either $\angle pbf$ or $\angle paf$ is a reflex angle.
In the former case we have a self charge, since $b$ is not connected to $p$ in $T(G)$.
In the latter, we have a self charge since $\angle paf$ is broken by $(a,c)$.
See Figures \ref{fi:Lemma43a.png}(d).
\end{proof}

\begin{figure}[h]
\centering
\includegraphics[scale=0.085]{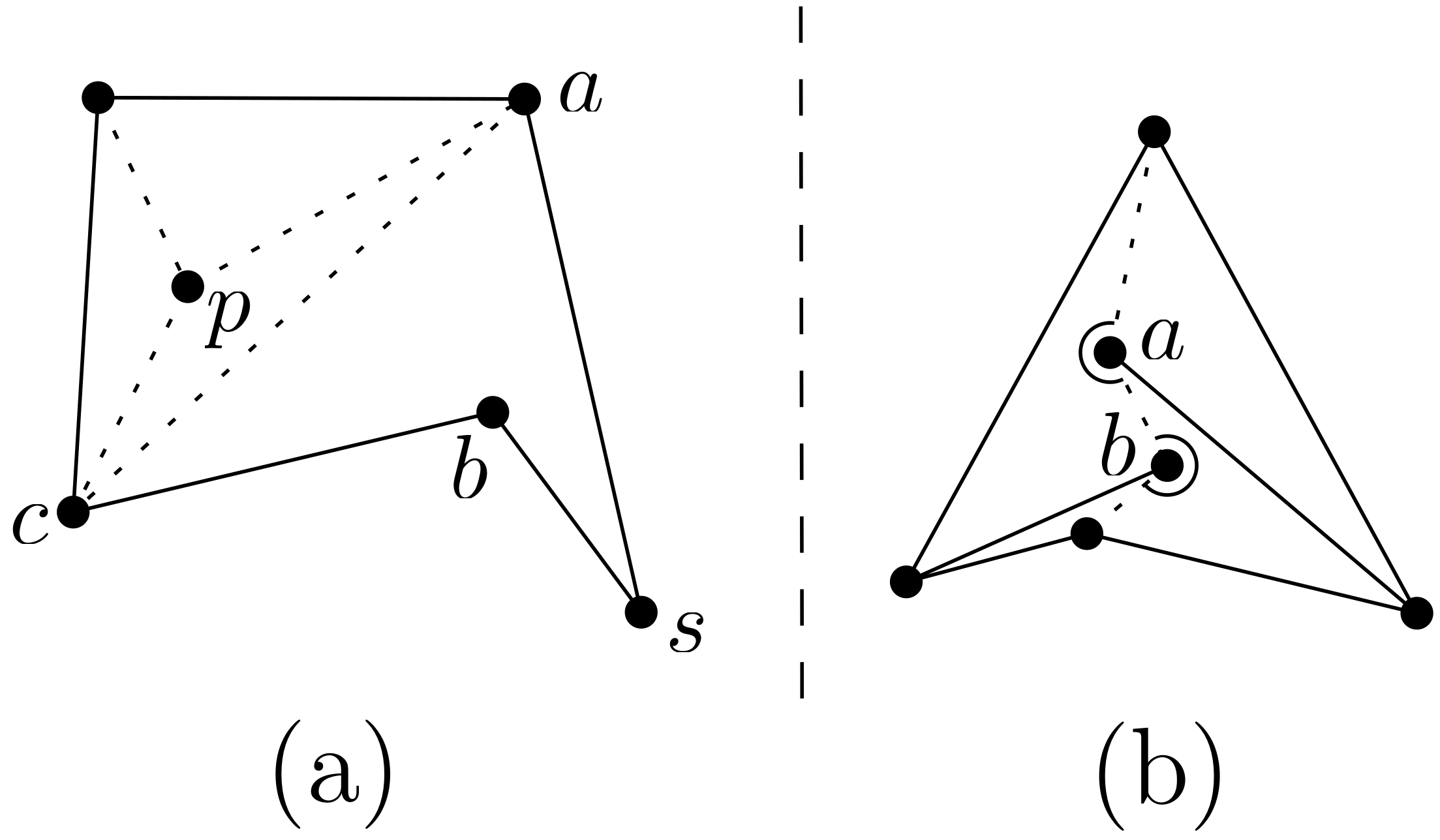}
\caption{The dotted edges are in $T(G)$ but not in $G$. (a) The potential 4 ving $p$ has a type (i) self charge corresponding to $(a,b)$. (b) The potential 4 vings $(a,G)$ and $(b,G)$ charge each other, leading instead to type (ii) self charges.}
\label{fi:TypeI}
\end{figure}

By Lemma \ref{le:RemRefSelf}(a), type (i) self charges cannot occur when $p$ has degree 4 in $T(G)$. 
That is, such self charges may only occur when a diagonal of the quadrilateral of $v$ appears in $T(G)$, leaving $p$ connected to three vertices of the quadrilateral. 
See Figure \ref{fi:TypeI}(a).
Then, type (i) self charges may correspond to the two edges of the quadrilateral of $v$ that are not part of the triangle that contains $p$ in $T(G)$.
In Figure \ref{fi:TypeI}(a), these are the edges $(a,b)$ and $(b,c)$.
This implies that every 4-ving has at most two type (i) self charges.

There is another issue that did not exist in the proof of Theorem \ref{th:potential}(a): A potential 4 ving may make a reflex charge to a ving of potential 3 or 4. 
Figure \ref{fi:TypeI}(b) depicts a problematic case where two potential 4 vings charge each other. 
If a potential 4 ving $(p,G)$ makes a reflex charge to a potential 3 or 4 ving, then we keep this charge on $(p,G)$. 
We refer to such a case as a \emph{self charge of type (ii)}. 

Due to self charges, a potential 4 ving $v$ may have charge left on itself after the second stage of charge redistribution. 
When that happens, we increase each of the charges that $v$ makes to be of size $7/8$ (or a smaller amount, if that suffices to remove all the remaining charge of $v$). 
After that, $v$ may still have remaining charge.

\parag{After the second stage of charge redistribution.} 
We now study the maximum charge that vings and hings may have after the above process.
We begin with high-potential vings. 

\begin{figure}[h]
\centering
\includegraphics[scale=0.11]{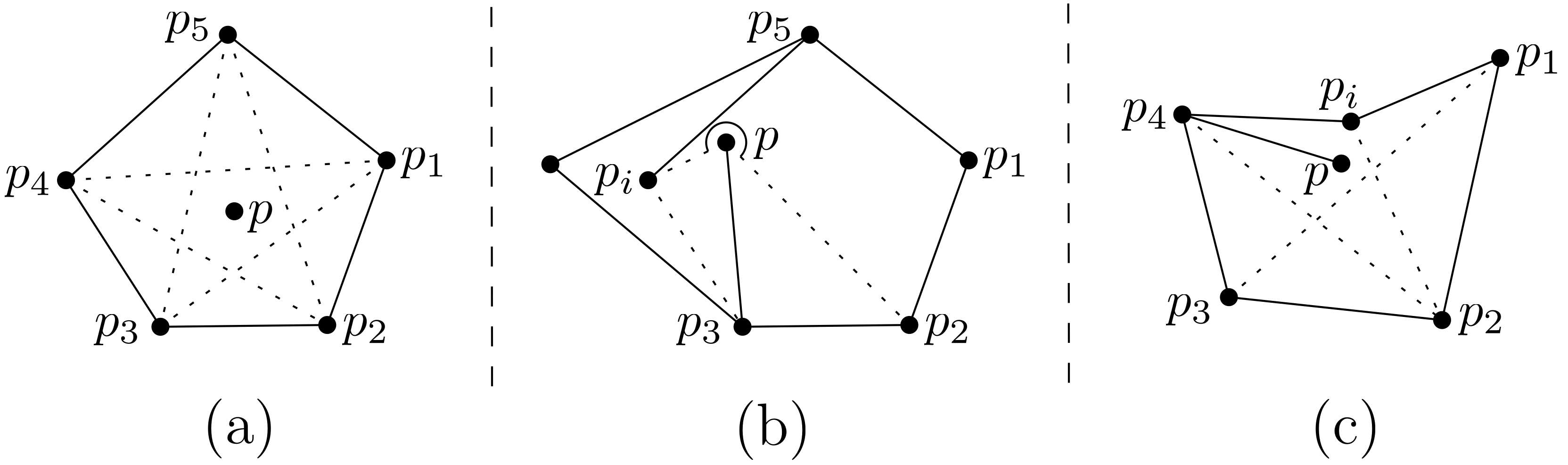}
\caption{ (a) Adding a diagonal gives $p$ potential 3 or 4. (b) The ving $(p_i,G)$ is affected by $p$ and by its two neighbors on the pentagon. (c) At most three diagonals can be added.}
\label{fi:Lemma52a}
\end{figure}

\begin{lemma} \label{le:case4HighPot}
A high-potential ving receives at most five charges from vings of potential 3 and 4. 
\end{lemma}
\begin{proof}
Consider a high potential ving $v=(p,G)$. 
We partition the analysis into cases.

\parag{The case of potential 5.}
Assume that $v$ has potential 5 and is affected by vertices $p_1,\ldots,p_5$.
Assume that $p_1,\ldots,p_5$ are numbered in their clockwise order around $p$, forming the vertices of a pentagon that contains $p$.
We note that every removal charge to $v$ corresponds to adding an edge between two vertices of $\{p_1,\ldots,p_5\}$, leading to $p$ having potential 3 or 4. 
Since such an edge is a diagonal of the above pentagon, $v$ receives at most five removal charges.
See Figure \ref{fi:Lemma52a}(a).

Lemma \ref{le:ReflexGen} states that $v$ receives at most two reflex charges.
Such a reflex charge must originate from a vertex of $p_1,\ldots,p_5$ in $G$. 
Consider a vertex $p_i$ that makes a reflex charge to $v$. 
We note that $(p_i,G)$ is affected by $p$ and by its two neighbors on the pentagon.
See Figure \ref{fi:Lemma52a}(b).
Since hings do not make charges, we may assume that $p_i$ is not a hull vertex.
We first consider the case where $p_i$ is a hull vertex of the pentagon $\{p_1,p_2,p_3,p_4,p_5\}$.
In this case, $p_i$ is also affected by at least one vertex outside of the pentagon. 
Thus, for $(p_i,G)$ to have potential 4, it must not be affected by the remaining two pentagon vertices.
In such a case, it is possible to add at most three diagonals of the pentagon to $G$. 
We conclude that, if $v$ receives a reflex charge from a hull vertex of the pentagon, then it receives at most three removal charges.
That is, $v$ receives at most five charges.

It remains to consider the case where $v$ receives a reflex charge from a $p_i$ that is not a hull vertex of the pentagon. 
In this case, the two neighbors of $p_i$ on the pentagon do not form a pentagon diagonal. 
See Figure \ref{fi:Lemma52a}(c).
For $p_i$ to have a potential of at most 4, at least one vertex of the pentagon does not affect $p_i$. 
That is, at most three internal diagonals of the pentagon could be added to $G$, as in Figure \ref{fi:Lemma52a}(c).
Once again, $(p,G)$ receives at most five charges.

\begin{figure}[h]
\centering
\includegraphics[scale=0.067]{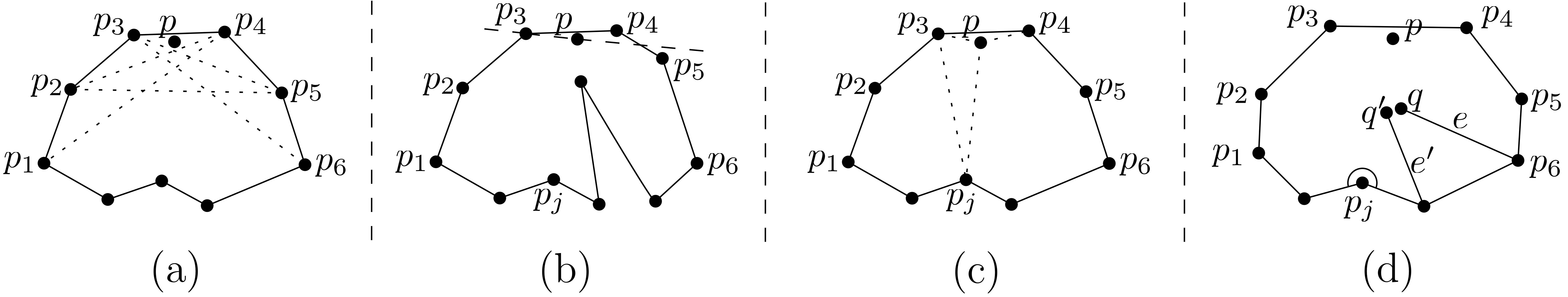}
\caption{Dotted edges are in $T(G)$ but not in $G$. (a) At most five diagonals may give $p$ a potential of 3 or 4. (b) The only point above the line through $p_3$ and $p$ is $p_4$. (c) A case where $T(G)$ contains $\triangle p p_3 p_j$. (d) The visibility between $p_j$ and $p_4$ is blocked by $e$. The visibility between $p_j$ and $q$ is blocked by $e'$. 
}\label{fi:Lemma52b}
\end{figure}

\parag{The case of potential larger than 5.} 
We now assume that $v$ has potential $k\ge 6$ and denote the vertices that affect $v$ as $p_1,\ldots,p_k$, according to their clockwise order around $p$.
A removal charge to $v$ corresponds to adding an edge between two vertices of $p_1,\ldots,p_k$ that decreases the potential of $p$ to 3 or 4. 
By observing Figure \ref{fi:Lemma52b}(a), we note that there are at most three edges whose addition leads to potential 4 and at most two edges whose addition leads to potential 3. 
Moreover, after adding such an edge, the 3 or 4 vertices that still affect $p$ are consecutive with respect to their order around $p$ in $G$.
That is, these vertices are of the form $p_i, p_{i+1}, p_{i+2},p_{i+3}$ (if an index exceeds $k$, we consider it $mod\ k$).
Thus, $v$ receives at most five removal charges.

Consider the case where $v$ receives five removal charges, as in Figure \ref{fi:Lemma52b}(a).
Without loss of generality, we may assume that the relevant vertices are $p_1,\ldots,p_6$.
We will show that, in this case, $v$ cannot receive reflex charges.

We note that $T(G)$ must contain the triangle $\triangle p p_3 p_4$. 
Moreover, the quadrilateral $p_2 p_3 p_4 p_5$ is a convex quadrilateral whose interior does not intersect any edges of $G$. 
This implies that each of $(p_2,G), (p_3,G), (p_4,G),(p_5,G)$ is either a hing or a high-potential ving.
Thus, if $(p,G)$ receives a reflex charge from $(p_j,G)$, then $j\neq 2,3,4,5$.

We claim that the only edges that may potentially form a reflex angle around $v$ are $(p,p_3)$ and $(p,p_4)$.
Indeed, for an edge $(p,p_i)$ with $i\neq 3,4$ to form a reflex angle with $(p,p_3)$, the vertex $p_i$ must be above the line through $p$ and $p_3$. 
See Figure \ref{fi:Lemma52b}(b).
By the above, $p_5$ is below this line. 
We note that every $p_i$ with $i>5$ (and $i<3$) must also be below this line, due to the order of these points around $p$. 
A symmetric argument holds when replacing $(p,p_3)$ with $(p,p_4)$.
We conclude that, when $v$ receives 5 removal charges, it receives no reflex charges.

For such a reflex charge to exist, $T(G)$ must contain the triangle $\triangle p p_3 p_j$ or $\triangle p p_4 p_j$.
Without loss of generality, we assume that $T(G)$ contains $\triangle p p_3 p_j$.
See Figure \ref{fi:Lemma52b}(c).
In $G$, the vertex $p$ may only be connected by an edge to $p_3$ and $p_4$. 
Any other edge that is connected to $p$ in $G$ would prevent it from receiving five removal charges.
For example, if the edge $(p,p_5)$ is in $G$, then we cannot add the diagonals $(p_1,p_4)$ and $(p_2,p_4$), so $v$ receives at most three removal charges.

The ving $(p_j,G)$ is affected by $p_3,p,p_{j-1},p_{j+1}$, so its potential is at least 4.
We may assume that the quadrilateral $p p_{j-1} p_j p_{j+1}$ has a reflex angle at $p_j$, since otherwise $(p_j,G)$ is a hing or a high-potential ving. 
See Figure \ref{fi:Lemma52b}(c).
If $(p_4,G)$ affects $(p_j,G)$ then $(p_j,G)$ is a high-potential ving.
We may thus assume that an edge $e$ blocks the visibility between $(p_4,G)$ and $(p_j,G)$, as in Figure \ref{fi:Lemma52b}(d). 
If an endpoint of $e$ is not $p_{j-1}$ and affects $v$, then we are again done. 
It is possible for another edge $e'$ to block the visibility from that endpoint.
See Figure \ref{fi:Lemma52b}(d).
If an endpoint of $e'$ is not $p_{j-1}$ and affects $v$, then we are again done. 
That endpoint of $e'$ may not be visible from $v$ due to yet another edge, and so on.
After a finite number of such steps, we must reach another ving that affects $v$.
Once again, this implies that $(p_j,G)$ is a high-potential ving. 
We conclude that $(p_j,G)$ is a hing or a high-potential ving, which contradicts the assumption that $(p_j,G)$ makes a reflex charge.

\begin{figure}[h]
\centering
\includegraphics[scale=0.108]{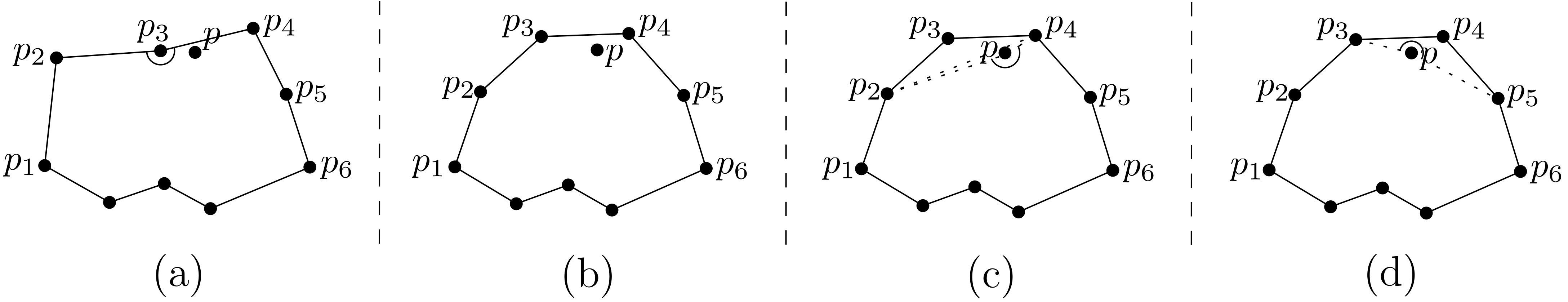}
\caption{Dotted edges are in $T(G)$ but not in $G$. (a) There is a single removal charge from a potential 3 ving, due to the reflex angle at $p_3$. 
(b) There is a single removal charge from a potential 3 ving, since $p$ is not in $\triangle p_2p_3p_4$. (c) The case where $(p_2,p_4)$ is in $T(G)$. (d) The case where $(p,p_3)$ is in $T(G)$. } 
\label{fi:Lemma52c}
\end{figure}

\parag{The case of 4 removal charges.}
We continue to study the case where $v$ has potential $k\ge 6$.
Since every  ving receives at most two reflex charges, if $v$ receives at most three removal charges then we are done.
Thus, it remains to consider the case where $v$ receives exactly four removal charges.
There are two ways for this to happen:
\begin{itemize}[noitemsep,topsep=1pt]
\item The quadrilateral $p_2p_3p_4p_5$ has a reflex angle at $p_3$ or $p_4$, so $v$ receives a single removal charge from a potential 3 ving. See Figure \ref{fi:Lemma52c}(a). 
\item The quadrilateral $p_2p_3p_4p_5$ is convex, but $p$ is outside of $\triangle p_2p_3p_4$ or $\triangle p_3p_4p_5$. Once again, $v$ receives a single removal charge from a potential 3 ving. See Figure \ref{fi:Lemma52c}(b).
\end{itemize}
In either case, $v$ still receives three removal charges from potential 4 vings.
As before, $p$ cannot be connected to any vertex in $G$, except possibly $p_3$ and $p_4$. 
Any other edge connected to $p$ would lead to at most three removal charges.

We begin with the former case. 
Without loss of generality, assume that $p_2p_3p_4p_5$ has a reflex angle at $p_3$, as in Figure \ref{fi:Lemma52c}(a).
We may repeat the above analysis for the case of five removal charges and $k\ge 6$ (that is, showing that no ving can make a reflex charge to $v$). 
In this case, $(p_3,G)$ remains a high-potential ving since it is affected by $(p,G),(p_2,G),(p_4,G),(p_5,G)$, $(p_6,G)$.
We conclude that, in this case, $v$ receives at most four charges.

Finally, we consider the case where the quadrilateral $p_2p_3p_4p_5$ is convex. 
Without loss of generality, we assume that $v$ is outside of $\triangle p_2p_3p_4$, as in Figure \ref{fi:Lemma52c}(b). 
We would like to repeat the same analysis, but this problem behaves a bit differently.
First, while the edge $(p,p_4)$ must exist in $T(G)$, this is no longer the case for $(p,p_3)$.   
Instead, $T(G)$ may contain the edges $(p_2,p_4)$ and $(p,p_2)$.
See Figure \ref{fi:Lemma52c}(c).
We partition our analysis according to whether $(p,p_3)$ or $(p_2,p_4)$ is in $T(G)$.

Assume that $(p,p_3)$ is in $T(G)$, as depicted in Figure \ref{fi:Lemma52c}(d).
In this case, we can repeat the above analysis for the case when $v$ receives five removal charges. 
The only difference is that $\angle p_5pp_3$ is a reflex angle, so $(p_5,G)$ may potentially send a reflex charge to $v$. 
However, we note that $(p_5,G)$ is either a hing or a high-potential ving, so this cannot occur after all.
Once again, we get that $v$ cannot receive reflex charges and $v$ is charged at most four times.

It remains to consider the case where $(p_2,p_4)$ is in $T(G)$.
In this case, we can again repeat the above analysis, but with $(p,p_2)$ taking the role of $(p,p_3)$. 
That is, the only edges that may form reflex angles around $p$ are $(p,p_2)$ and $(p,p_4)$.
See Figure \ref{fi:Lemma52c}(c).
Repeating the above implies that any $p_j$ that forms a triangle $\triangle pp_2p_j$ in $T(G)$ corresponds either to a hing or to a high-potential ving. 
\end{proof}

We are now ready to establish the maximum charge that a high potential ving may have.

\begin{lemma}
After all charge redistribution steps, every high potential ving has a charge of at most $37/8$.
\end{lemma}
\begin{proof}
Consider a high potential ving $v$.
Lemma \ref{le:case4HighPot} states that $v$ receives charge from at most five potential 3 and 4 vings. 
When $v$ has potential at least 6, it receives charge from at most two potential 3 vings (see the proof of this case in Lemma \ref{le:case4HighPot}). 
Thus, $v$ has a charge of at most $3\cdot 7/8 + 2\cdot 1 = 37/8$.

It remains to consider the case where $v$ has potential 5. 
If $v$ receives at most four charges then it has a charge of at most $4\cdot 1 = 4$.
We may thus assume that $v$ receives exactly five charges.

We consider the case where $v$ receives five removal charges and no reflex charges.
That is, the pentagon of $v$ is convex.
By considering Figure \ref{fi:Lemma52a}(a) and trying different positions for $p$, we note that at most two of the removal charges to $v$ come from potential 3 vings.
In other words, no matter where we place $p$, there are at most two diagonals whose addition gives $p$ potential 3.
Thus, as above, the charge of $v$ is at most $37/8$.

Finally, we consider the case where $v$ receives at least one reflex charge.
By this part of the proof of 
Lemma \ref{le:case4HighPot}, such a charge implies that there are at most three removal charges, which contradicts our assumption of five charges.
The same argument implies that two reflex charges to $v$ yield at most two removal charges.
We conclude that this case cannot occur.
\end{proof}

We now study the maximum charge a hing may have. 

\begin{lemma} \label{le:case4Hing}
After the two charge redistribution steps, every hing holds a charge of at most 2.
\end{lemma}
\begin{proof}
Consider a hing $h=(p,G)\in \pts_H\times \gr(\pts)$.
By definition, $h$ receives no removal charges.
By Lemma \ref{le:ReflexGen}, the edges that form a reflex angle around $h$ are the two edges of the convex hull of $\pts$.
Each of these two edges corresponds to at most one reflex charge.
Indeed, one side of each edge is outside the convex hull of $\pts$ and cannot lead to such charges.
We conclude that $h$ is charged at most twice, receiving a charge of at most $2\cdot 1 = 2$. 
\end{proof}

Finally, we study the charge that remains on potential 4 vings due to self charges.

\begin{lemma} \label{le:case34Charge4}
Each potential 4 ving ends up with a charge of at most $7/4$.
\end{lemma}
\begin{proof}
We consider a potential 4 ving $v=(p,G)$.
Let $p_1,p_2,p_3,p_4$ be the vertices of the quadrilateral of $v$, in clockwise order around $p$.
We partition the analysis into cases, as follows.

\begin{figure}[h]
\centering
\includegraphics[scale=0.064]{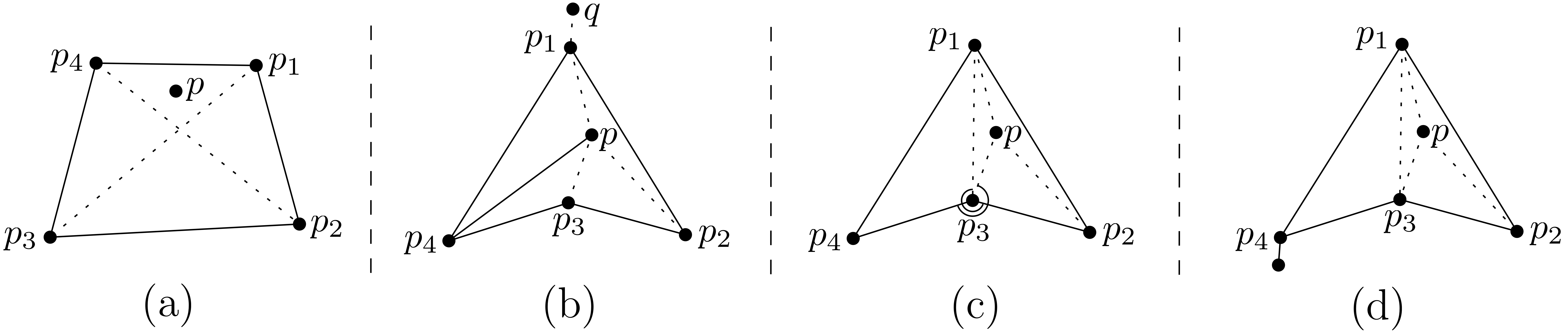}
\caption{(a) A potential 4 ving that receives two removal charges from potential 3 vings. 
(b) A type (ii) self charge from potential 3 ving $(p_3,G)$. 
(c) Two type (ii) self charges from potential 4 ving $(p_3,G)$.
(d) The edge $(p_1,p_4)$ corresponds to a type (i) self charge due to the outside reflex angle around $p_4$.
 }
\label{fi:Cases12}
\end{figure}

\newcounter{cases}

\stepcounter{cases}
\emph{\underline{Case \arabic{cases}.} The ving $v$ receives two removal charges from potential 3 vings.}
In this case, the quadrilateral of $v$ is convex, and all vertices of this quadrilateral are visible from each other in $G$.
See Figure \ref{fi:Cases12}(a).
This implies that each vertex of the quadrilateral is either a hing or a high potential vertex.
That is, there are no type (ii) self charges.

We recall that $v$ can have at most two type (i) self charges. 
Thus, $v$ remains with a charge of at most $7/2-2\cdot 7/8 = 7/4$.
In the following cases, we assume that $v$ receives at most one removal charge.

\stepcounter{cases}
\emph{\underline{Case \arabic{cases}.} The ving $v$ has a type (ii) self charge due to a potential 3 ving. }
Without loss of generality, we assume that the potential 3 ving is $(p_3,G)$. See Figure \ref{fi:Cases12}(b).
Since $(p_3,G)$ is affected by $(p,G)$, $(p_2,G)$, and $(p_4,G)$, we get that $(p_3,G)$ is not affected by another vertex.
In particular, $(p_3,G)$ is not affected by $(p_1,G)$ due to $(p,p_2)$ or $(p,p_4)$.
Also, $p_3$ must be a reflex vertex of the quadrilateral of $v$.
Then, $(p,G)$ has two type (ii) self charges due to $(p_3,G)$, corresponding to $(p_2,p_3)$ and to $(p_3,p_4)$.

In this case, $p$ is connected to all of $p_1,p_2,p_3,p_4$ in $T(G)$.
By Lemma \ref{le:RemRefSelf}(a), there are no self charges of type (i).
We note that each of $(p_2,G)$ and $(p_4,G)$ is either a hing or a high potential ving.
Similarly, $(p_1,G)$ is either a hing or a ving of potential at least 4.
This implies that $v$ receives a single charge from potential 3 ving: a reflex charge from $(p_3,G)$. 
That is, after the first step of charge redistribution, $(p,G)$ has a charge of $5/2$.

By the above, $v$ may have type (ii) self charges due to $(p_1,G)$, but not due to $(p_2,G)$ and $(p_4,G)$. 
For $(p_1,G)$ to be a potential 4 ving, it should be affected by exactly one additional vertex $q$. 
For $p_1$ not to be a hull vertex, $(q,G)$ should be affected by both $(p_2,G)$ and $(p_4,G)$.
See Figure \ref{le:RemRefSelf}(b).
If one of the edges $(p_4,p_1)$ and $(p_1,p_2)$ corresponds to a type (ii) self charge of $v$, then the other edge corresponds to a removal charge from $v$. 
We conclude that $v$ remains with a charge of at most $5/2-7/8 = 13/8$. 

In all following cases, we assume that there are no type (ii) self charges due to potential 3 vings. 
Since $v$ receives no reflex charges from potential 3 vings and at most one removal charge, $v$ distributes a charge of at most $5/2$.

\stepcounter{cases}
\emph{\underline{Case \arabic{cases}.} The ving $v$ has two self charges due to the same potential 4 ving. }
Without loss of generality, we assume that the ving corresponding to these two charges is $(p_3,G)$.
Then $p_3$ is a reflex vertex of the quadrilateral of $v$, and $(p_3,G)$ is affected by exactly $(p,G),(p_1,G),(p_2,G),(p_4,G)$.
See Figure \ref{fi:Cases12}(c).
The edges $(p_2,p_3)$ and $(p_3,p_4)$ correspond to self charges of $v$.

In this case, each of $(p_1,G),(p_2,G),(p_4,G)$ is a hing or a high potential ving.
Thus, these cannot correspond to a type (ii) self charge.
Without loss of generality, we assume that $p$ is in $\triangle p_1p_2p_3$ (rather than in $\triangle p_3p_4p_1$).
The only possible third self charge of $v$ is a type (i) self charge with $(p_1,p_4$).
See Figure \ref{fi:Cases12}(d).
Thus, $v$ ends up with a charge of at most $5/2-7/8 = 13/8$.

In the following cases, if the quadrilateral has a reflex vertex, then this vertex is a high potential ving.

\begin{figure}[h]
\centering
\includegraphics[scale=0.14]{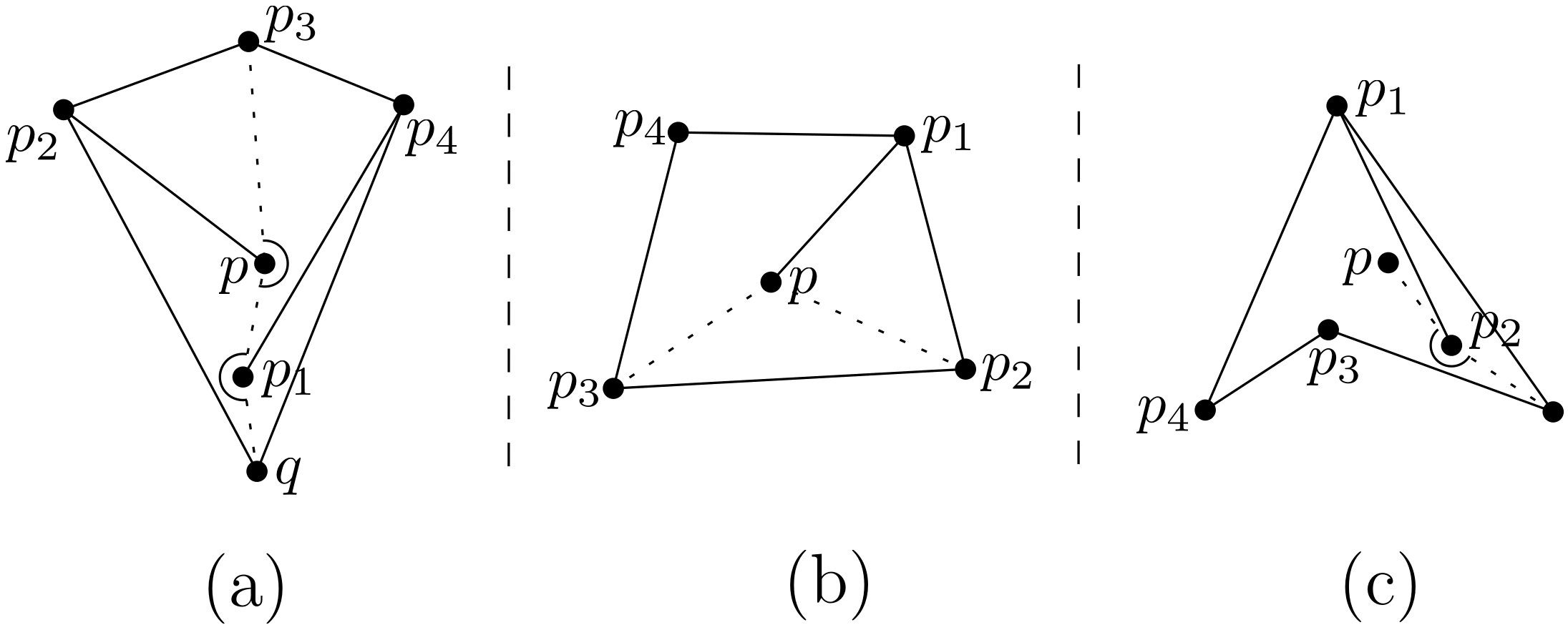}
\caption{Dotted edges are in $T(G)$ but not in $G$. (a) A type (i) self charge due to $(p_1,G)$. (b) The quadrilateral of $v$ is convex. We cannot add $(p_2,p_4)$ due to the existence of $(p,p_1)$. (c) The quadrilateral of $v$ is not convex.
}
\label{fi:Case4}
\end{figure}

\stepcounter{cases}
\emph{\underline{Case \arabic{cases}.} In $T(G)$, the vertex $p$ has degree 4.}
That is, $p$ is connected to $p_1,p_2,p_3,p_4$ in $T(G)$.
By Lemma \ref{le:RemRefSelf}(a), $v$ has no type (i) self charges.
We claim that $v$ has at most two type (ii) self charges, so it remains with a charge of at most $5/2-2\cdot 7/8=3/4$.

By the above, $v$ may not receive type (ii) self charges from reflex vertices of the quadrilateral of $v$. 
Without loss of generality, we assume that $v$ has a type (ii) self charge due to $(p_1,G)$. 
Since $(p_1,G)$ must be a potential 4 ving, there exists $q$ that is affected by $(p_1,G)$, $(p_2,G)$, and $(p_3,G)$.
See Figure \ref{fi:Case4}(a).

We note that if one of the edges $(p_1,p_2)$ and $(p_4,p_1)$ corresponds to a type (ii) self charge of $v$, then the other edge corresponds to a removal charge from $v$. 
In addition, each of $(p_2,G)$ and $(p_4,G)$ is a hing or a high potential ving. 
Thus, $v$ has no self charges due to $(p_2,G)$ and $(p_4,G)$.
A second type (ii) self charge may symmetrically correspond to $(p_3,G)$.

\stepcounter{cases}
\emph{\underline{Case \arabic{cases}.} In $T(G)$, the vertex $p$ has degree 3.}
Without loss of generality, we assume that $p$ is connected to $p_1,p_2,p_3$ in $T(G)$.
That is, $T(G)$ includes the edge $(p_1,p_3)$ and $p$ is in $\triangle p_1p_2p_3$.
In this case, $(p,G)$ may have two type (i) self charges, corresponding to the edges $(p_3,p_4)$ and $(p_4,p_1)$.
Since $(p_1,p_3)\in T(G)$, both $p_2$ and $p_4$ cannot be reflex vertices of the quadrilateral of $v$.
See Figures \ref{fi:Case4}(b,c).

In this case, it is possible that both $(p_3,p_4)$ and $(p_4,p_1)$ correspond to type (i) self charges.
By repeating the analysis of the preceding case, at most one of $(p_1,p_2)$ and $(p_2,p_3)$ corresponds to a type (ii) self charge. 
Since $v$ has at most three self charges, it has a remaining charge of at most $5/2-7/8=13/8$.
\end{proof}

Since we established the maximum charge each ving and hing may receive, the proof is complete.
\end{proof}

\end{document}